\documentclass[11pt]{amsart}
\usepackage{amsmath,amsthm,amsfonts,amssymb,mathrsfs, mathtools}
\usepackage{enumerate}
\usepackage[colorlinks=true, linkcolor=blue, citecolor=magenta, menucolor=black]{hyperref}
\usepackage[demo]{graphicx}
\usepackage[up]{caption}
\usepackage{subcaption}
\usepackage{pgf,tikz}
\usepackage[toc,page]{appendix}
\usetikzlibrary{arrows}
\usetikzlibrary{patterns}
\usepackage{color}

\numberwithin{equation}{section}
\theoremstyle{plain}
\newtheorem{thm}{Theorem}[section]

\newtheorem{lem}[thm]{Lemma}
\newtheorem{prop}[thm]{Proposition}

\newtheorem{letterthm}{Theorem}

\newtheorem{lettercor}[letterthm]{Corollary}

\theoremstyle{definition}
\newtheorem{defn}[thm]{Definition}
\newtheorem{facts}[thm]{Facts}
\newtheorem*{defn*}{Definition}

\newtheorem{rem}[thm]{Remark}

\newtheorem{claim}[thm]{Claim}
\newcommand{\N}{\mathbb{N}}
\newcommand{\R}{\mathbb{R}}
\newcommand{\C}{\mathbb{C}}
\newcommand{\Z}{\mathbb{Z}}

\newcommand{\Tr}{\operatorname{Tr}}

\newcommand{\Stab}{\operatorname{Stab}}

\newcommand{\ovt}{\mathbin{\overline{\otimes}}}
\newcommand{\Aut}{\operatorname{Aut}}
\newcommand{\Ad}{\operatorname{Ad}}

\newcommand{\id}{\operatorname{id}}
\newcommand{\GL}{\operatorname{GL}}

\newcommand{\SL}{\operatorname{SL}}

\newcommand{\Prob}{\operatorname{Prob}}

\newcommand{\Sub}{\operatorname{Sub}}
\newcommand{\rk}{\operatorname{rk}}

\newcommand{\dpr}{^{\prime\prime}}

\newcommand{\Haar}{\operatorname{Haar}}

\newcommand{\rW}{\operatorname{W}}
\newcommand{\rM}{\operatorname{M}}

\newcommand{\rE}{\operatorname{ E}}
\newcommand{\rC}{\operatorname{C}}

\newcommand{\rL}{\operatorname{L}}

\newcommand{\dom}{\operatorname{dom}}
\newcommand{\Lie}{\operatorname{Lie}}
\newcommand{\Gr}{\operatorname{Gr}}
\newcommand{\CL}{\operatorname{CL}}

\allowdisplaybreaks

\makeatletter
\@namedef{subjclassname@2020}{%
  \textup{2020} Mathematics Subject Classification}
\makeatother

\begin{document}

\title[Strong primeness for equivalence relations]{Strong primeness for equivalence relations arising from Zariski dense subgroups}

\begin{abstract}
We show that orbit equivalence relations arising from essentially free ergodic probability measure preserving actions of Zariski dense discrete subgroups of simple algebraic groups are strongly prime. As a consequence, we prove the existence and the uniqueness of a prime factorization for orbit equivalence relations arising from direct products of higher rank lattices. This extends and strengthens Zimmer's primeness result for equivalence relations arising from actions of lattices in simple Lie groups. The proof of our main result relies on a combination of ergodic theory of algebraic group actions and Popa's intertwining theory for equivalence relations.
\end{abstract}

\author{Daniel Drimbe}
\address{Department of Mathematics \\ University of Iowa \\
14 MacLean Hall \\ Iowa City 52242 \\ USA}
\email{daniel-drimbe@uiowa.edu}
\thanks{DD was supported by Engineering and Physical Sciences Research Council grant EP/X026647/1}

\author{Cyril Houdayer}
\address{\'Ecole normale sup\'erieure \\ D\'epartement de math\'ematiques et applications \\ Universit\'e Paris-Saclay \\ 45 rue d'Ulm \\ 75230 Paris Cedex 05 \\ France}
\email{cyril.houdayer@ens.psl.eu}
\thanks{CH is supported by Institut Universitaire de France and ERC Advanced Grant NET 101141693}


\maketitle

\section{Introduction and statement of the main results}

Throughout, by an equivalence relation $\mathscr R$, we simply mean a probability measure preserving (pmp) countable Borel equivalence relation defined on a standard probability space $(X, \nu)$.

\subsection*{Introduction}

We say that a type ${\rm II_1}$ ergodic equivalence relation $\mathscr R$ is {\em prime} when for any ergodic equivalence relations $\mathscr T_1, \mathscr T_2$ for which $\mathscr R \cong \mathscr T_1 \times \mathscr T_2$, there exists $i \in \{1, 2\}$ such that $\mathscr T_i$ has finite orbits almost everywhere. Zimmer's seminal work \cite{Zi81} shows that for any lattice $\Gamma < G$ in a noncompact simple connected real Lie group with finite center and any essentially free ergodic pmp action $\Gamma \curvearrowright (X, \nu)$, the orbit equivalence relation $\mathscr R(\Gamma \curvearrowright X)$ is prime. 

In order to motivate and state our main results, we introduce the following terminology using the framework of Popa's intertwining theory for equivalence relations (see Section \ref{preliminaries} for further details). We say that a type ${\rm II_1}$ ergodic equivalence relation $\mathscr R$ is {\em strongly prime} when for any ergodic equivalence relations $\mathscr S$, $\mathscr T_1$, $\mathscr T_2$ for which $\mathscr R \times \mathscr S \cong \mathscr T_1 \times \mathscr T_2$, there exists $i \in \{1, 2\}$ such that as subequivalence relations, $\mathscr T_i$ embeds into $\mathscr S$ inside $\mathscr R \times \mathscr S$, that we write $\mathscr T_i \preceq_{\mathscr R \times \mathscr S} \mathscr S$ (see Definition \ref{defn-strong-primeness} for a precise statement). Note that this implies in particular that there exists an ergodic equivalence relation $\mathscr U$ such that $\mathscr T_i \times \mathscr U$ and $\mathscr S$ are stably isomorphic (see \cite{Sp21}). We point out that the notion of strong primeness for type ${\rm II_1}$ factors was studied by Isono in \cite{Is16}.

The ergodic theory of equivalence relations is intimately connected to the theory of von Neumann algebras. By a well-known construction due to Feldman--Moore \cite{FM75}, to any ergodic equivalence relation $\mathscr R$ on a standard probability space $(X, \nu)$, one can associate a von Neumann factor $\rL(\mathscr R)$ so that $\rL^\infty(X, \nu) \subset \rL(\mathscr R)$ sits as a Cartan subalgebra. Recall that a type ${\rm II_1}$ factor $M$ is {\em prime} when for any factors $M_1, M_2$ for which $M \cong M_1 \ovt M_2$, there exists $i \in \{1, 2\}$ such that $M_i$ is finite dimensional. For any ergodic equivalence relation $\mathscr R$, if $\rL(\mathscr R)$ is prime, then $\mathscr R$ is prime, but the converse need not hold in general (see Proposition \ref{prop-counterexample} below). In relation with Zimmer's primeness result \cite{Zi81}, in the case of simple connected real Lie groups with finite center and real rank one, any discrete subgroup $\Gamma < G$ is biexact in the sense of \cite{BO08}. Then by Ozawa's relative solidity result \cite{Oz04}, it follows that for any essentially free ergodic pmp action $\Gamma \curvearrowright (X, \nu)$ of a countable discrete nonamenable biexact group, the group measure space factor $\rL(\Gamma \curvearrowright X)$ is prime. For other iconic primeness results in von Neumann algebra theory, we refer to \cite{Ge96, Oz03, Pe06}.

Over the last two decades, Popa's deformation/rigidity theory (see the ICM surveys \cite{Po06, Va10, Io18}) has led to a plethora of primeness and indecomposability results for von Neumann algebras and orbit equivalence relations arising from essentially free ergodic pmp actions of {\em negatively curved groups} such as free groups, hyperbolic groups, biexact groups, groups with a positive first $\ell^2$-betti number and so on (see also \cite{Ad92, Ki05}). Except for very few results (see e.g.\! \cite{DHI16, BIP18, IM19}), higher rank lattices remain allergic to deformation/rigidity theory methods.

The key novelty of our paper is to combine methods from ergodic theory of algebraic group actions and Popa's intertwining theory to obtain strong primeness and existence and uniqueness of prime factorization for orbit equivalence relations arising from essentially free ergodic pmp actions of higher rank lattices and more generally Zariski dense discrete subgroups of simple algebraic groups.

\subsection*{Statement of the main results}

Our first main theorem is the following strong primeness result strengthening Zimmer's result \cite[Theorem 1.1]{Zi81}.

\begin{letterthm}\label{main-theorem}
Let $k$ be a local field of characteristic zero, $\mathbf G$ a $k$-isotropic almost $k$-simple algebraic $k$-group and $\Gamma < \mathbf G(k)$ a Zariski dense discrete subgroup. Let $\Gamma \curvearrowright (X, \nu)$ be an essentially free ergodic pmp action. 

Then the orbit equivalence relation $\mathscr R(\Gamma \curvearrowright X)$ is strongly prime.
\end{letterthm}

By Borel's density theorem \cite{Bo60}, Theorem \ref{main-theorem} applies to all essentially free ergodic pmp actions of lattices $\Gamma < \mathbf G(k)$, where $\mathbf G$ is a $k$-isotropic almost $k$-simple algebraic $k$-group (e.g.\! $\SL_n(\Z) < \SL_n(\R)$ for every $n \geq 2$). 

In that respect, denote by $\mathfrak Z$ the class of orbit equivalence relations $\mathscr R = \mathscr R(\Gamma \curvearrowright X)$ that arise from essentially free ergodic pmp actions $\Gamma \curvearrowright (X, \nu)$, where $\Gamma < \mathbf G(k)$ is a Zariski dense discrete subgroup, $k$ is a local field of characteristic zero and $\mathbf G$ is a $k$-isotropic almost $k$-simple algebraic $k$-group. 

Using Theorem \ref{main-theorem}, we deduce the following {\em unique prime factorization} result for direct products of ergodic equivalence relations that belong to $\mathfrak Z$.

\begin{lettercor}\label{main-cor}
Let $n \geq 2$ and $\mathscr R_1, \dots, \mathscr R_n$ be ergodic equivalence relations that belong to $\mathfrak Z$. Then the following assertions hold:

\begin{itemize}

\item [$(\rm i)$] If $\mathscr R = \mathscr R_1 \times \cdots \times \mathscr R_n \cong \mathscr T_1 \times \mathscr T_2$, where $\mathscr T_1, \mathscr T_2$ are type ${\rm II_1}$ ergodic equivalence relations, then there exists a partition $\{1, \dots, n\} = T_1 \sqcup T_2$ into nonempty subsets such that  for every $j \in \{1, 2\}$, we have $\mathscr T_j \preceq_{\mathscr R} \mathscr R_{T_j}$ as subequivalence relations, where $\mathscr R_{T_j} = \prod_{i \in T_j} \mathscr R_i$.

\item [$(\rm ii)$] If $\mathscr R = \mathscr R_1 \times \cdots \times \mathscr R_n \cong \mathscr T_1 \times \cdots \times \mathscr T_p$, where $\mathscr T_1, \dots, \mathscr T_p$ are type ${\rm II_1}$ ergodic equivalence relations and $p \geq n$, then $n = p$ and upon permuting the indices, $\mathscr T_i$ and $\mathscr R_i$ are stably isomorphic for every $1 \leq i \leq n$. 

\end{itemize}
\end{lettercor}

We point out that the first unique prime factorization result appeared in von Neumann algebra theory in the pioneering work of Ozawa--Popa \cite{OP03}. Indeed, they showed that tensor products of type ${\rm II_1}$ factors arising from countable discrete icc nonamenable biexact groups have a unique tensor product decomposition upon taking amplifications and permuting the indices.

Assuming that the equivalence relations $\mathscr R_1, \dots, \mathscr R_n$ that belong to $\mathfrak Z$ are moreover strongly ergodic, we obtain the following sharper unique prime factorization result.

\begin{lettercor}\label{main-cor-sharp}
Let $n \geq 2$ and $\mathscr R_1, \dots, \mathscr R_n$ be strongly ergodic equivalence relations that belong to $\mathfrak Z$. Then the following assertions hold:

\begin{itemize}

\item [$(\rm i)$] If $\mathscr R = \mathscr R_1 \times \cdots \times \mathscr R_n \cong \mathscr T_1 \times \mathscr T_2$, where $\mathscr T_1, \mathscr T_2$ are type ${\rm II_1}$ ergodic equivalence relations, then there exist a partition $\{1, \dots, n\} = T_1 \sqcup T_2$ into nonempty subsets and $t > 0$ such that 
$$\mathscr T_1^{t} \cong \mathscr R_{T_1} \quad  \text{and} \quad  \mathscr T_2^{1/t} \cong \mathscr R_{T_2}.$$

\item [$(\rm ii)$] If $\mathscr R = \mathscr R_1 \times \cdots \times \mathscr R_n \cong \mathscr T_1 \times \cdots \times \mathscr T_p$, where $\mathscr T_1, \dots,  \mathscr T_p$ are type ${\rm II_1}$ ergodic equivalence relations and $p \geq n$, then $n = p$ and there exist $t_1, \dots, t_n > 0$ such that $t_1 \cdots t_n = 1$ and upon permuting the indices, $\mathscr T_i^{t_i} \cong \mathscr R_i$ for every $1 \leq i \leq n$. 

\end{itemize}
\end{lettercor}

The statement of Corollary \ref{main-cor-sharp}$(\rm ii)$ is sharper than the one in Corollary \ref{main-cor}$(\rm ii)$ in the sense that we may choose $t_1, \dots, t_n > 0$ such that $\mathscr T_i^{t_i} \cong \mathscr R_i$ for every $1 \leq i \leq n$ in such a way that $t_1 \cdots t_n = 1$. 

In particular, by Borel's density theorem \cite{Bo60} and Kazhdan's theorem \cite{Ka66}, Corollary \ref{main-cor-sharp} applies to direct products of equivalence relations arising from essentially free ergodic pmp actions of lattices $\Gamma < \mathbf G(k)$, where $\mathbf G$ is an almost $k$-simple algebraic $k$-group such that $\rk_k(\mathbf G) \geq 2$ (e.g.\! $\SL_n(\Z) < \SL_n(\R)$ for every $n \geq 3$). Corollary \ref{main-cor-sharp} should be compared with Hoff's unique prime factorization result for strongly ergodic equivalence relations with nontrivial $1$-cohomology \cite{Ho15}. Let us mention that Isono--Marrakchi \cite{IM19} obtained a unique prime factorization result for tensor products of group measure space factors arising from {\em compact} pmp actions of higher rank lattices with Kazhdan's property (T). In that respect, Corollaries \ref{main-cor} and \ref{main-cor-sharp} provide the first unique prime factorization result for orbit equivalence relations arising from {\em arbitrary} pmp actions of higher rank lattices.

Our second main theorem gives a complete characterization of primeness for orbit equivalence relations arising from essentially free ergodic pmp actions of product groups $\Gamma_1 \times \cdots \times \Gamma_n \curvearrowright (X, \nu)$, where $\Gamma_i$ is a discrete group as in Theorem \ref{main-theorem} with Kazhdan's property (T).

\begin{letterthm}\label{main-theorem-decomposition}
Let $n \geq 1$. For every $1 \leq i \leq n$, let $k_i$ be a local field of characteristic zero, $\mathbf G_i$ a $k_i$-isotropic $k_i$-simple algebraic $k_i$-group and $\Gamma_i < \mathbf G_i(k_i)$ a Zariski dense discrete subgroup with Kazhdan's property {\em (T)}. Set $\Gamma = \Gamma_1 \times \cdots \times \Gamma_n$. Let $\Gamma \curvearrowright (X, \nu)$ be an essentially free ergodic pmp action and set $\mathscr R = \mathscr R(\Gamma \curvearrowright X)$. Assume that $\mathscr R \cong \mathscr T_1 \times \mathscr T_2$, where $\mathscr T_1, \mathscr T_2$ are type ${\rm II_1}$ ergodic equivalence relations.

Then there exist a partition $\{1, \dots, n\} = T_1 \sqcup T_2$ into nonempty subsets, $t > 0$, an essentially free ergodic pmp action $\Gamma_{T_j} \curvearrowright (X_j, \nu_j)$, where $\Gamma_{T_j} = \prod_{i \in T_j} \Gamma_i$ for every $j \in \{1, 2\}$, such that the following assertions hold: 
\begin{itemize}
\item [$(\rm i)$] $\Gamma \curvearrowright (X, \nu) \cong \Gamma_{T_1} \times \Gamma_{T_2} \curvearrowright (X_1 \times X_2, \nu_1 \otimes \nu_2)$ and
\item [$(\rm ii)$] $\mathscr T_1^t \cong \mathscr R(\Gamma_{T_1} \curvearrowright X_1)$ and  $\mathscr T_2^{1/t} \cong\mathscr R(\Gamma_{T_2} \curvearrowright X_2)$.
\end{itemize}
\end{letterthm}

Theorem \ref{main-theorem-decomposition} is an analogue of Drimbe's primeness characterization for group measure space factors arising from essentially free ergodic pmp actions of products of hyperbolic groups with Kazhdan's property (T) \cite{Dr19}. By Borel's density theorem \cite{Bo60} and Kazhdan's theorem \cite{Ka66}, Theorem \ref{main-theorem-decomposition} applies to all essentially free ergodic pmp actions $\Gamma \curvearrowright (X, \nu)$, where $\Gamma_i < \mathbf G_i(k_i)$ is a lattice and $\mathbf G_i$ is a $k_i$-simple algebraic $k_i$-group such that $\rk_{k_i}(\mathbf G_i) \geq 2$ for every $i \in \{1, \dots, n\}$. 

As a consequence of Theorem \ref{main-theorem-decomposition}, we obtain the following existence and uniqueness of a prime factorization for equivalence relations.

\begin{lettercor}\label{main-cor-decomposition}
Let $n \geq 1$. For every $1 \leq i \leq n$, let $k_i$ be a local field of characteristic zero, $\mathbf G_i$ a $k_i$-isotropic $k_i$-simple algebraic $k_i$-group and $\Gamma_i < \mathbf G_i(k_i)$ a Zariski dense discrete subgroup with Kazhdan's property {\em (T)}. Set $\Gamma = \Gamma_1 \times \cdots \times \Gamma_n$. Let $\Gamma \curvearrowright (X, \nu)$ be an essentially free ergodic pmp action.

Then there exist a unique partition $\{1, \dots, n\} = T_1 \sqcup \cdots \sqcup T_r$ into nonempty subsets (upon permuting the indices), an essentially free ergodic pmp action $\Gamma_{T_j} \curvearrowright (X_j, \nu_j)$, where $\Gamma_{T_j} = \prod_{i \in T_j} \Gamma_i$ for every $j \in \{1, \dots, r\}$, such that the following assertions hold:
\begin{itemize}
\item [$(\rm i)$] $\Gamma \curvearrowright (X, \nu) \cong \Gamma_{T_1} \times \cdots \times \Gamma_{T_r} \curvearrowright (X_1 \times \cdots \times X_r, \nu_1 \otimes \cdots \otimes \nu_r)$ and
\item [$(\rm ii)$] $\mathscr R(\Gamma_{T_j} \curvearrowright X_j)$ is prime for every $j \in \{1, \dots, r\}$.
\end{itemize}

\end{lettercor}

\subsection*{Comments on the proofs}

The proof of Theorem \ref{main-theorem} builds upon Zimmer's proof of \cite[Theorem 1.1]{Zi81} and combines ergodic theory of algebraic group actions and Popa's intertwining theory for equivalence relations. Unlike Zimmer's proof of \cite[Theorem 1.1]{Zi81}, we do not use induction and we work directly with the orbit equivalence relation $\mathscr R(\Gamma \curvearrowright X)$. This allows us to deal with essentially free ergodic pmp action of {\em arbitrary} Zariski dense discrete subgroups $\Gamma < \mathbf G(k)$ rather than lattices $\Gamma < \mathbf G(k)$ as in \cite{Zi81}. The proof of Corollary \ref{main-cor} follows by combining Theorem \ref{main-theorem} with Popa's intertwining theory.

Using Tucker-Drob's results on the structure of inner amenable groups \cite{TD14}, we observe in Proposition \ref{prop-not-inner-amenable} that any Zariski dense discrete subgroup $\Gamma < \mathbf G(k)$ as in Theorem \ref{main-theorem} is not inner amenable. This implies that for any essentially free strongly ergodic pmp action $\Gamma \curvearrowright (X, \nu)$, the group measure space von Neumann factor $\rL(\Gamma \curvearrowright X)$ is {\em full} meaning that it has no nontrivial central sequences (see \cite{Ch81}). For the proof of Corollary \ref{main-cor-sharp}, we use Theorem \ref{main-theorem} and we exploit Popa's intertwining theory in the setting of von Neumann algebras in combination with Isono--Marrakchi's results on tensor product decompositions of full factors \cite{IM19}.

The proof of Theorem \ref{main-theorem-decomposition} and Corollary \ref{main-cor-decomposition} follows the proof of Theorem \ref{main-theorem} and moreover exploits intertwining techniques from \cite{Dr19}.

\subsection*{Acknowledgments} This work was initiated when CH was visiting the Mathematical Institute of the University of Oxford during Hilary Term 2024. He is grateful towards Stuart White for his kind invitation. We also thank Adrian Ioana for his valuable comments.

%{
%  \hypersetup{linkcolor=black}
%  \tableofcontents
%}
%

\section{Preliminaries}\label{preliminaries}

\subsection{Equivalence relations}

Let $\mathscr R$ be an equivalence relation defined on a standard probability space $(X, \nu)$. Denote by $\mathscr B(X)$ (resp.\! $\mathscr B(\mathscr R)$) the $\sigma$-algebra of all Borel subsets of $X$ (resp.\! $\mathscr R$). For every $x \in X$, we denote by $[x]_{\mathscr R}$ the $\mathscr R$-equivalence class of $x \in X$. We define the $\sigma$-finite Borel measure $m$ on $\mathscr R$ by the formula
$$\forall \mathscr W \in \mathscr B(\mathscr R), \quad m (\mathscr W) = \int_{X} \sharp (\left\{(x, y) \mid y \in [x]_{\mathscr R} \right\}\cap \mathscr W) \, {\rm d}\nu(x).$$
We denote by $\Aut(X, \nu)$ the group of all pmp Borel automorphisms $\theta : (X, \nu) \to (X, \nu)$. We denote by $\Aut(\mathscr R)$ the automorphism group of $\mathscr R$ that consists of all pmp Borel automorphisms $\theta \in \Aut(X, \nu)$ such that $(\theta(x), \theta(y)) \in \mathscr R$ and $(\theta^{-1}(x), \theta^{-1}(y)) \in \mathscr R$ for $m$-almost every $(x, y) \in \mathscr R$. Then we denote by $[\mathscr R] < \Aut(\mathscr R)$ the {\em full group} of $\mathscr R$ that consists of all automorphisms $\theta \in \Aut(\mathscr R)$ for which $(\theta(x), x) \in \mathscr R$ for $\nu$-almost every $x \in X$. The uniform metric 
$$d_u : [\mathscr R] \times [\mathscr R] \to \R_+ : (\theta, \rho) \mapsto \nu(\{x \in X \mid \theta(x) \neq \rho(x)\})$$
is complete and separable and so the full group $[\mathscr R]$ is a Polish group. By \cite{FM75}, there exists a countable subgroup $\Lambda < [\mathscr R]$ such that $\mathscr R = \mathscr R(\Lambda \curvearrowright X)$ is the orbit equivalence relation of the pmp action $\Lambda \curvearrowright (X, \nu)$. The {\em full pseudogroup} of $\mathscr R$ consists of all pmp Borel partial isomorphisms $\theta : U \to V$, where $U, V \in \mathscr B(X)$ and $(\theta(x), x) \in \mathscr R$ for $\nu$-almost every $x \in U$. For every $\theta : U \to V \in [[\mathscr R]]$, there exists $\rho \in [\mathscr R]$ such that $\rho|_U = \theta$. We say that $\mathscr R$ is {\em ergodic} if and only if the Koopman representation $\kappa : [\mathscr R] \to \mathscr U(\rL^2(X, \nu)^0)$ is ergodic, where $\rL^2(X, \nu)^0 = \rL^2(X, \nu) \ominus \C \mathbf 1_X$. 

Whenever $U \subset X$ is a nonnull measurable subset, we define the equivalence relation $\mathscr R|_U = \mathscr R \cap (U \times U)$ on the standard probability space $(U, \nu_U)$ where $\nu_U = \frac{1}{\nu(U)} \nu|_U$. Assume that $\mathscr R$ is ergodic and that $(X, \nu)$ is diffuse. For every $t > 0$, we define the amplification $\mathscr R^t$ as follows. Denote by $c$ the counting measure on $\N$. Choose a measurable subset $X_t \subset X \times \N$ such that $(\nu \otimes c)(X_t) = t$ and set $\nu_t = \frac{1}{(\nu \otimes c)(X_t)} (\nu \otimes c)|_{X_t}$. Define the ergodic equivalence relation $\mathscr R^t$ on the standard probability space $(X_t, \nu_t)$ by declaring $((x, p), (y, q)) \in \mathscr R^t$ if and only if $(x, y) \in \mathscr R$ for $((x, p), (y, q)) \in X_t \times X_t$.
 
For every $n \in \N \cup \{+\infty\}$, define the Borel subset 
$$X_n = \left \{ x \in X \mid \sharp ([x]_{\mathscr R}) = n \right \} \subset X.$$
We say that $\mathscr R$ is (essentially) {\em finite} if $\nu(X_\infty) = 0$. We say that $\mathscr R$ is (essentially) {\em bounded} if there exists $k \in \N$ such that $\nu(X_n) = 0$ for every $n \geq k$.

Let $\mathscr S$ be another equivalence relation defined on a standard probability space $(Y, \eta)$. We say that $\mathscr R$ and $\mathscr S$ are {\em isomorphic} if there exists a pmp Borel isomorphism $\theta : (X, \nu) \to (Y, \eta)$ such that for almost every $(x, y) \in \mathscr R$, we have $(\theta(x), \theta(y)) \in \mathscr S$ and for almost every $(x, y) \in \mathscr S$, we have $(\theta^{-1}(x), \theta^{-1}(y)) \in \mathscr R$. Assuming moreover that $\mathscr R$ and $\mathscr S$ are both ergodic, we say that $\mathscr R$ and $\mathscr S$ are {\em stably isomorphic} if there exist nonnull measurable subsets $U \subset X$ and $V \subset Y$ and a pmp Borel isomorphism $\theta : (U, \nu_U) \to (V, \eta_V)$ such that for almost every $(x, y) \in \mathscr R|_U$, we have $(\theta(x), \theta(y)) \in \mathscr S|_V$ and for almost every $(x, y) \in \mathscr S|_V$, we have $(\theta^{-1}(x), \theta^{-1}(y)) \in \mathscr R|_U$.

Let $\mathscr S \leq \mathscr R$ be a subequivalence relation. For every $n \in \N \cup \{+\infty\}$, define the Borel subset 
$$Y_n = \left \{ x \in X \mid [x]_{\mathscr R} \text{ is the union of } n \; \mathscr S\text{-classes} \right \} \subset X.$$
We say that $\mathscr S \leq \mathscr R$ has (essentially) {\em finite index} if $\nu(Y_\infty) = 0$. We say that $\mathscr S \leq \mathscr R$ has (essentially) {\em bounded index} if there exists $k \in \N$ such that $\nu(Y_n) = 0$ for every $n \geq k$.

Following \cite{FM75}, we denote by $\rL(\mathscr R)$ the tracial von Neumann algebra associated with $\mathscr R$. Then $\rL^\infty(X) \subset \rL(\mathscr R)$ is a Cartan subalgebra. We simply denote by $\left\{u_\theta \mid \theta \in [\mathscr R] \right\} \subset \mathscr U(\rL(\mathscr R))$ the full group of $\mathscr R$ regarded as a subgroup of the unitary group $\mathscr U(\rL(\mathscr R))$. Then we have 
$$\rL(\mathscr R) = \left\{u_\theta \mid \theta \in [\mathscr R] \right\}\dpr \vee \rL^\infty(X).$$

\begin{facts}\label{fact-ergodic}
Keep the same notation as above. We will be using the following useful properties.
\begin{itemize}
    \item [$(\rm i)$] The equivalence relation $\mathscr R$ is ergodic if and only if the Koopman representation $\kappa : [\mathscr R] \to \mathscr U(\rL^2(X,\nu)^0)$ is weakly mixing, where $\rL^2(X, \nu)^0 = \rL^2(X, \nu) \ominus \C \mathbf 1_X$. 

    \item [$(\rm ii)$] If $\mathscr R$ is ergodic and $(X, \nu)$ is diffuse, then $\rL(\mathscr R) = \left\{u_\theta \mid \theta \in [\mathscr R] \right\}\dpr$. 
    
    \item [$(\rm iii)$] Let $\Psi \in \Aut(X, \nu)$ be a pmp Borel automorphism and denote by $\mathscr R_\Psi$ the equivalence relation generated by $\Psi$. Then for any $\varepsilon>0$, there exist a measurable subset $X_0\subset X$ with $\nu(X\setminus X_0)\leq \varepsilon$, a positive integer $p$ and $\Psi_0\in [\mathscr R_\Psi]$ such that $\Psi_0=\Psi$ on $X_0$ and $\Psi_0^p= \id_X$.
\end{itemize}
\end{facts}

\begin{proof}
$(\rm i)$ This follows from \cite[Theorem 11.20]{LM24}.

$(\rm ii)$ Indeed, it suffices to prove that $\rL^\infty(X) \subset \left\{u_\theta \mid \theta \in [\mathscr R] \right\}\dpr$. Let $Y \subset X$ be a measurable subset such that $\nu(Y) < 1$. Set $Z = X \setminus Y$ and $\mathscr S = \mathscr R|_Z$. Since $\mathscr S$ is ergodic and $(Z, \nu_Z)$ is diffuse, there exists an ergodic type ${\rm II_1}$ hyperfinite subequivalence relation $\mathscr T \leq \mathscr S$ (see e.g.\! \cite[Proposition 9.3.2]{Zi84}). By \cite{CFW81},  choose a free ergodic pmp action $\Z \curvearrowright (Z, \nu_Z)$ such that $\mathscr T = \mathscr R(\Z \curvearrowright Z)$. Denote by $\rho \in [\mathscr T]$ the generator of $\Z \curvearrowright (Z, \nu_Z)$. For every $n \in \N$, define $\theta_n \in [\mathscr R]$ by $\theta_n(x) = x$ if $x \in Y$ and $\theta_n(x) = \rho^n(x)$ if $x \in Z$. Then it is plain to see that $u_{\theta_n} \to \mathbf 1_Y$ weakly as $n \to \infty$. This shows that $\rL(\mathscr R) = \left\{u_\theta \mid \theta \in [\mathscr R] \right\}\dpr$.

$(\rm iii)$ Since $\mathscr R_\Psi$ is hyperfinite by Rokhlin's lemma, it follows that there exists an increasing sequence of finite equivalence relations $\mathscr T_1\leq \mathscr T_2\leq \dots$ such that  (up to a conull Borel subset) we have $\mathscr R_{\Psi}= \bigcup_{k\geq 1} \mathscr T_k$. Then for $\nu$-almost every $x\in X$, we have $[x]_{\mathscr R_\Psi}=\bigcup_{k\ge 1}[x]_{\mathscr T_k}$, and thus, there exists $k\ge 1$ such that $\Psi(x)\in [x]_{\mathscr T_k}$. Next, define the measurable sets $X_k=\{ x\in X \mid \Psi(x) \in [x]_{\mathscr T_k}  \}$, $k\ge 1$ and note that $X=\bigcup_{k\ge 1} X_k$. Take $N\ge 1$ such that $\nu(X\setminus X_N)\leq \varepsilon$ and let $\Psi_0\in [\mathscr T_N]\leq [\mathscr R_{\Psi}]$ be such that $\Psi_0=\Psi$ on $X_N$. By construction, it follows that the equivalence relation generated by $\Psi_0$ is finite, hence, the lemma follows by letting $X_0=X_N$. 
\end{proof}

\subsection{Popa's intertwining theory for equivalence relations}

We review Popa's criterion for intertwining von Neumann subalgebras \cite{Po01, Po03}. Let $(M, \tau)$ be a tracial von Neumann algebra and $A\subset 1_A M 1_A$, $B \subset 1_B M 1_B$ be von Neumann subalgebras. Let $\mathcal G\subset \mathscr U(A)$ be a subgroup such that $\mathscr G\dpr=A.$ By \cite[Corollary 2.3]{Po03} and \cite[Theorem A.1]{Po01} (see also \cite[Proposition C.1]{Va06}), the following conditions are equivalent:

\begin{itemize}
\item [$(\rm i)$] There exist $d \geq 1$, a projection $q \in \rM_d(B)$, a nonzero partial isometry $v \in \rM_{1, d}(1_A M)q$ and a unital normal $\ast$-homomorphism $\pi : A \to q\rM_d(B)q$  such that $a v = v \pi(a)$ for all $a \in A$.

\item [$(\rm ii)$] There is no net of unitaries $(w_n)_n$ in $\mathscr G$ such that
$$\forall x, y \in 1_A M 1_B, \quad \lim_n \|\rE_B(x^* w_n y)\|_2 = 0.$$

\end{itemize}
If one of the previous equivalent conditions is satisfied, we say that $A$ {\it embeds into} $B$ {\it inside} $M$ and write $A \preceq_M B$.
Moreover, if  the inclusion $B\subset M$ is unital, then the above are also equivalent to:
\begin{itemize}
    \item [$(\rm iii)$]  The unitary representation $\pi: \mathscr G\to \mathscr U(\rL^2(\langle M, e_B\rangle, \Tr))$ given by $\pi_u (\xi)= u \xi u^*$ for $u \in \mathscr G$, $\xi \in \rL^2(\langle M, e_B\rangle, \Tr)$, is not weakly mixing. Here, we denoted by ${\rm Tr}$ the associated semifinite trace on Jones'\! basic construction $\langle M, e_B\rangle$.
\end{itemize}

If for every nonzero projection $p \in A' \cap 1_A M 1_A$ we have $Ap \preceq_M B$, then we write $A \preceq^s_M B$.

Popa's criterion for intertwining von Neumann subalgebras was adapted to the setting of equivalence relations by Ioana \cite{Io11}. Let $\mathscr R$ be an equivalence relation defined on a standard probability space $(X, \nu)$. Let $Y, Z \subset X$ be nonnull measurable subsets and $\mathscr S \leq \mathscr R|_Y$ and $\mathscr T \leq \mathscr R|_Z$ be subequivalence relations. Following \cite{IKT08}, define the map 
$$\varphi_{\mathscr T} : [[\mathscr R]] \to [0, 1] : \theta \mapsto \nu(\{x \in \dom(\theta) \mid (\theta(x), x) \in \mathscr T \}).$$
By \cite[Lemma 1.7]{Io11} (see also \cite[Lemma 3.1]{Sp21}), the following conditions are equivalent:

\begin{itemize}
\item [$(\rm i)$] There exist a $\mathscr S$-invariant nonnull measurable subset $Y_0 \subset Y$ and a subequivalence relation $\mathscr S_0 \leq \mathscr S$ such that for any nonnull measurable subset $Y_1 \subset Y_0$, there is a nonnull measurable subset $U \subset Y_1$ and $\theta \in [[\mathscr R]]$ with $\theta : U \to V$, such that 
\begin{itemize}
\item $\mathscr S_0|_U \leq \mathscr S|_U$ has bounded index, and

\item  $(\theta \times \theta)(\mathscr S_0|_U) \leq \mathscr T|_V$.
\end{itemize}

\item [$(\rm ii)$] There is no sequence $(\theta_n)_n$ in $[\mathscr S]$ such that
$$\forall \psi, \rho \in [[\mathscr R]], \quad \lim_n \varphi_{\mathscr T}(\psi \theta_n \rho) = 0.$$
\end{itemize}
If one of the previous equivalent conditions is satisfied, we say that $\mathscr S$ {\it embeds into} $\mathscr T$ {\it inside} $\mathscr R$ and write $\mathscr S \preceq_{\mathscr R} \mathscr T$.

\begin{facts}\label{fact-equivalence}
Keep the same notation as above. We will be using the following useful properties. 
\begin{itemize}
\item [$(\rm i)$] We have $\mathscr S \preceq_{\mathscr R} \mathscr T$ if and only $\rL(\mathscr S) \preceq_{\rL(\mathscr R)} \rL(\mathscr T)$. 

\item [$(\rm ii)$] If $\mathscr S_0 \leq \mathscr S$ is a subequivalence relation and $\mathscr S \preceq_{\mathscr R} \mathscr T$, then we have $\mathscr S_0 \preceq_{\mathscr R} \mathscr T$.
\end{itemize}

For every $i \in \{1, 2\}$, let $(Z_i,\nu_i)$ be a standard probability space such that $(X,\nu)\cong (Z_1\times Z_2, \nu_1\otimes\nu_2)$. Let $Z \subset Z_1$ be a nonnull measurable subset and consider $\nu_Z = \frac{1}{\nu_1(Z)}\nu_1|_Z$. Let $\mathscr S$ be an ergodic equivalence relation on $(Z, \nu_Z)$ such that $\mathscr S\times \id_{Z_2}\leq \mathscr R|_{Z \times Z_2}$. Denote by $\kappa : [\mathscr S] \to \mathscr U(\rL^2(Z,\nu_Z)^0)$ the Koopman representation. Let $\mathscr T \leq \mathscr R$ be a  subequivalence relation.

\begin{itemize}
\item [$(\rm iii)$] Assume that $\mathscr S \times \id_{Z_2} \npreceq_{\mathscr R}  \mathscr T$. Then for all $\varepsilon>0$, $n \geq 1$, $f_1,\dots, f_n\in \rL^2(Z,\nu_Z)^{0}$ and $\psi_1,\dots,\psi_n\in [\mathscr R]$, there exists $\theta\in [\mathscr S]$ such that $|\langle \kappa_{\theta}(f_i),f_j \rangle|\leq\varepsilon$ and $\varphi_{\mathscr T}(\psi_i (\theta\times \id_{Z_2}) \psi_j)\leq \varepsilon$ for all $1\leq i,j\leq n$. 
\end{itemize}

For every $i \in \{1, 2\}$, let $\mathscr R_i$ be an ergodic equivalence relation on $(X_i, \nu_i)$ and $\mathscr S_i$ an ergodic equivalence relation on $(Y_i, \eta_i)$. Assume that $\mathscr R_1 \times \mathscr S_1 \cong \mathscr R_2 \times \mathscr S_2$. Set $(Z, \zeta) = (X_1 \times Y_1, \nu_1 \otimes \eta_1) \cong (X_2 \times Y_2, \nu_2 \otimes \eta_2)$ and $\mathscr R = \mathscr R_1 \times \mathscr S_1 \cong \mathscr R_2 \times \mathscr S_2$. We simply write $\mathscr R_1 \preceq_{\mathscr R} \mathscr R_2$ as subequivalence relations if $\mathscr R_1 \times \id_{Y_1} \preceq_{\mathscr R} \mathscr R_2 \times \id_{Y_2}$.

\begin{itemize}
\item  [$(\rm iv)$] If $\mathscr S_1 \preceq_{\mathscr R} \mathscr S_2$ as subequivalence relations, then we have $$\rL^\infty(X_2) \preceq_{\rL(\mathscr R)} \rL^\infty(X_1).$$

\item [$(\rm v)$] If $\mathscr T_1 \leq \mathscr S_1$ is an ergodic subequivalence relation and $\mathscr T_1 \preceq_{\mathscr R} \mathscr S_2$ as subequivalence relations, then we have $\mathscr S_1 \preceq_{\mathscr R} \mathscr S_2$ as subequivalence relations. 

\item [$(\rm vi)$] If $\mathscr R_1 \preceq_{\mathscr R} \mathscr R_2$ as subequivalence relations, then there exists an ergodic equivalence relation $\mathscr V$ such that $\mathscr R_1 \times \mathscr V$ and $\mathscr R_2$ are stably isomorphic.
\end{itemize}
\end{facts}

\begin{proof}
$(\rm i)$ This follows from \cite[Lemma 1.8]{Io11}.

$(\rm ii)$ It is obvious from the definition.

$(\rm iii)$ Since $\mathscr S$ is ergodic, it follows by Facts \ref{fact-ergodic}$(\rm i), (\rm ii)$ that $\rL(\mathscr S) = [\mathscr S]\dpr$ and that the Koopman representation $\kappa : [\mathscr S] \to \mathscr U(\rL^2(Z,\nu_Z)^0)$ is weakly mixing. Denote by $\left\{u_\theta \mid \theta \in [\mathscr R] \right\} \subset \mathscr U(\rL(\mathscr R))$ the full group of $\mathscr R$ regarded as a subgroup of the unitary group $\mathscr U(\rL(\mathscr R))$.
By assumption, we have $\rL(\mathscr S) \ovt \rL^\infty(Z_2) \npreceq_{\rL(\mathscr R)} \rL(\mathscr T)$, which implies that there exist two sequences of unitaries $(v_n)_{n\ge1}\subset \mathscr U(\rL (\mathscr S))$ and $(w_n)_{n\ge 1}\subset \mathscr U(\rL^\infty(Z_2))$ such that $\| \rE_{\rL(\mathscr T)} (x v_n w_n u_\theta) \|_2\to 0$, for all $x\in \rL(\mathscr R)$ and $\theta\in [\mathscr R]$. Since $u_\theta^* w_nu_\theta = w_n\circ \theta \in \mathscr U(\rL^\infty (X))$, for any $n\ge 1$, and $\rL^\infty(X)\subset \rL(\mathscr T)$, it follows that $\| \rE_{\rL(\mathscr T)} (x v_n a u_\theta) \|_2\to 0$, for all $x\in \rL(\mathscr R)$, $a\in \rL^\infty(X)$ and $\theta\in [\mathscr R]$. 
This implies that $\rL(\mathscr S) \npreceq_{\rL(\mathscr R)} \rL(\mathscr T)$. This further implies that the unitary representation $\pi: [\mathscr S]\to \mathscr U(\rL^2(\langle \rL(\mathscr R), e_{\rL(\mathscr T)}\rangle , \Tr))$ given by $\pi_u (\xi)= u \xi u^*$ for $u \in [\mathscr S]$, $\xi \in \rL^2(\langle \rL(\mathscr R), e_{\rL(\mathscr T)}\rangle , \Tr)$, is weakly mixing. Since $\|\rE_{\rL(\mathscr T)}(u_{\theta})\|_2^2=\varphi_{\mathscr T}(\theta)$ for any $\theta\in [\mathscr R]$ and since the direct sum $\kappa\oplus \pi$ is a weakly mixing unitary representation of $[\mathscr S]$ as well, the conclusion follows.

$(\rm iv)$ This follows from \cite[Corollary 3.4]{Sp21}

$(\rm v)$ If $\mathscr T_1 \leq \mathscr S_1$ is an ergodic subequivalence relation and $\mathscr T_1 \preceq_{\mathscr R} \mathscr S_2$ as subequivalence relations, then we have $\rL^\infty(X_1) \ovt \rL(\mathscr T_1) \preceq_{\rL(\mathscr R)} \rL^\infty(X_2) \ovt \rL(\mathscr S_2)$ and so $\rL^\infty(X_2) \preceq_{\rL(\mathscr R)} \rL^\infty(X_1)$ by \cite[Lemma 3.5]{Va07}. By applying again \cite[Lemma 3.5]{Va07}, we infer that $\rL^\infty(X_1) \ovt \rL(\mathscr S_1) \preceq_{\rL(\mathscr R)} \rL^\infty(X_2) \ovt \rL(\mathscr S_2)$ and so $\mathscr S_1 \preceq_{\mathscr R} \mathscr S_2$ as subequivalence relations.

$(\rm vi)$ This follows from \cite[Proposition 3.6]{Sp21}.
\end{proof}

The following result and its proof are inspired by \cite[Corollary F.14]{BO08}. It will turn out to be useful in the proof of Theorem \ref{main-theorem-decomposition}.

\begin{lem}\label{lem-key-intertwining}
Let $\mathscr R$ be an equivalence relation on $(X, \nu)$. For every $i\in\{1,2\}$, let $(Z_i,\nu_i)$ be a standard probability space such that $(X,\nu)\cong (Z_1\times Z_2, \nu_1\otimes\nu_2)$. Let $\mathscr S$ be an ergodic equivalence relation on $(Z_1, \nu_1)$ such that $\mathscr S\times \id_{Z_2}\leq \mathscr R$ and $\mathscr T \leq \mathscr R$ a subequivalence relation. 

If $\mathscr S \times \id_{Z_2} \npreceq_{\mathscr R}  \mathscr T$, then there exists a hyperfinite ergodic subequivalence relation $\mathscr S_0 \leq \mathscr S$ such that  $\mathscr S_0 \times \id_{Z_2} \npreceq_{\mathscr R}  \mathscr T$. 
\end{lem}

\begin{proof}  Choose a countable dense subset $\{g_n \mid n \geq 1\}$ in $[\mathscr R]$ with respect to the uniform metric $d_u$. Let $(f_n)_{n\ge 1}$ be a dense sequence in $\rL^2(Z_1,\nu_1)^{0}$.  
For any measurable subset $Z\subset Z_1$, we set $\nu_Z=\frac{1}{\nu_1(Z)}\nu_1|_{Z}$ and we let $\kappa^{Z}: [\mathscr S|_{Z}]\to \mathscr U(\rL^2(Z,\nu_Z)^{0})$ be the associated Koopman representation. The proof of the lemma will be concluded by constructing an increasing sequence $(\mathscr S_n)_n$ of bounded subequivalence relations of $\mathscr S$ and a sequence $(\theta_n)_n$ in $[\mathscr S]$ such that for all $n \geq 1$ and all $1\leq i,j\leq n$, we have $\theta_n \in [\mathscr S_n]$ and
\begin{equation*}
\varphi_{\mathscr T}(g_i (\theta_n\times \id_{Z_2}) g_j)\leq \frac{1}{n} \quad \text{and} \quad |\langle \kappa_{\theta_n}(f_i),f_j  \rangle|\leq \frac{1}{n}.
\end{equation*}

Let $\mathscr S_1 \leq \mathscr S$ be the trivial subequivalence relation, $\theta_1 = \id_{Z_1}$ and assume that $\mathscr S_1, \dots, \mathscr S_{n-1}$ and $\theta_1,\dots,\theta_{n-1}$ have been constructed. 
Since $\mathscr S_{n-1}$ is bounded, there exists a finite measurable partition $Z_1=X_1 \sqcup \dots \sqcup X_{N}$ into $\mathscr S_{n-1}$-invariant sets such that $\mathscr S_{n-1}|_{X_{k}}$ is of type ${\rm I}_{t_k}$ for any $1\leq k\leq N$. For any $1\leq k\leq N$, let $Y_k\subset X_k$ be a fundamental domain for $\mathscr S_{n-1}|_{X_k}$ and let $\omega_k\in[\mathscr S_{n-1}|_{X_k}]$ such that $X_k=\sqcup_{r=0}^{t_k-1}\omega_k^r Y_k$.

Let $1\leq k\leq N$. Since $\mathscr S|_{Y_k} \times \id_{Z_2} \npreceq_{\mathscr R}  \mathscr T$, Facts \ref{fact-equivalence}$(\rm iii)$ implies that there exists $\Psi_k\in [\mathscr S|_{Y_k}]$ such that for all $0\leq r\leq t_k-1$ and $1\leq i,j\leq n$, we have
\begin{equation}\label{erg1}
\varphi_{\mathscr T}(g_i  (\omega_k^r\Psi_k \omega_k^{-r}\times \id_{Z_2})g_j)\leq \frac{1}{2n t_k N} \quad     
\end{equation}
and 
\begin{equation}\label{erg111}
\quad |\langle \kappa^{Y_k}_{\Psi_k}((f_i\circ \omega_k^r)|_{Y_k}),(f_j\circ\omega_k^r)|_{Y_k}  \rangle|\leq \frac{1}{n}. 
\end{equation}
Next, by applying Facts \ref{fact-ergodic}$(\rm iii)$, there exist a measurable subset $\tilde Y_k\subset Y_k$ with $\nu_1(Y_k\setminus \tilde Y_k)\leq \frac{1}{2n t_k N}$, a positive integer $p_k$ and $\tilde\Psi_k\in [\mathscr S|_{Y_k}]$ such that $\tilde \Psi_k=\Psi_k$ on $\tilde Y_k$ and $\tilde \Psi_k^{p_k}=\id_{Y_k}$. Fix $0\leq r\leq t_{k-1}$. Since $g_i (\omega_k^r\tilde\Psi_k \omega_k^{-r}\times \id_{Z_2})g_j=g_i (\omega_k^r\Psi_k \omega_k^{-r}\times \id_{Z_2} )g_j$ on $g_j^{-1} (\omega_k^r \tilde Y_k\times Z_2)$, it follows that for all  $1\leq i,j\leq n$, we have
\begin{align*}
& \varphi_{\mathscr T}(g_i (\omega_k^r\tilde\Psi_k \omega_k^{-r} \times \id_{Z_2}) g_j) \leq \nu((Y_k\setminus \tilde Y_k)\times Z_2)+\\ 
& \nu(\{ x \in g_j^{-1} (\tilde Y_k \times Z_2)\mid  ((g_i (\omega_k^r\Psi_k \omega_k^{-r}\times \id_{Z_2}) g_j)(x),x) \in \mathscr T \}),
\end{align*}
and thus, using \eqref{erg1} we have
\begin{equation}\label{erg112}
 \varphi_{\mathscr T}(g_i (\omega_k^r\tilde\Psi_k \omega_k^{-r} \times \id_{Z_2}) g_j)\leq\frac{1}{nt_kN}.
\end{equation}

Next, since $Y_k\subset X_k$ is a fundamental domain for $\mathscr S_{n-1}|_{X_k}$, we can extend $\tilde\Psi_k\in [\mathscr S|_{Y_k}]$ naturally to an element $\Omega_k\in [\mathscr S|_{X_k}]$. More precisely, we let $\Omega_k(\omega_k^r y)= \omega_k^r \tilde\Psi_k(y)$, for all $y\in Y_k$ and $0\leq r\leq t_k-1$.
In this way, since $\nu_{X_k}(Y_k)=\frac{1}{t_k}$, for all $1\leq i,j\leq n$, we deduce from \eqref{erg112} and \eqref{erg111} that
\begin{equation}\label{erg2}
\varphi_{\mathscr T}(g_i (\Omega_k\times \id_{Z_2}) g_j)\leq \frac{1}{n N} \quad \text{and} \quad |\langle \kappa^{X_k}_{\Omega_k}(f_i|_{X_k}),f_j|_{X_k}  \rangle|\leq \frac{1}{n}.    
\end{equation}
Indeed, note that \eqref{erg112} implies that
\begin{align*}
& \varphi_{\mathscr T}(g_i (\Omega_k\times \id_{Z_2}) g_j) \\
&=\sum_{r=0}^{t_k-1} \nu(\{ x \in g_j^{-1} (\omega_k^r Y_k \times Z_2)\mid ((g_i (\Omega_k\times \id_{Z_2}) g_j)(x),x) \in \mathscr T \}) \\
&=\sum_{r=0}^{t_k-1} \nu(\{ x \in g_j^{-1} (\omega_k^r Y_k \times Z_2)\mid ((g_i (\omega_k^r\tilde\Psi_k \omega_k^{-r}\times \id_{Z_2}) g_j)(x),x) \in \mathscr T \}) \\
&\leq \sum_{r=0}^{t_k-1}\varphi_{\mathscr T}(g_i (\omega_k^r\tilde\Psi_k \omega_k^{-r} \times \id_{Z_2}) g_j)\leq\frac{1}{nN}.
\end{align*}
The second part of \eqref{erg2} follows in a similar way and we omit the details.

Next, define $\theta_n\in \Aut(Z_1, \nu_1)$ by $\theta_n= \Omega_k$ on $X_k$ for any $1\leq k\leq N$. By construction, we deduce that $\theta_n\in [\mathscr S]$ and that there exists an integer $p$ (e.g.\! take $p=p_1 \cdots p_N$) for which $\theta_n^p= \id_{Z_1}$. Note that for all $1\leq i,j\leq n$, it follows from \eqref{erg2} that 
\begin{align*}
& \varphi_{\mathscr T}(g_i (\theta_n\times \id_{Z_2}) g_j) \\
&=\nu(\bigcup_{k=1}^{N} \{ x \in g_j^{-1} (X_k\times Z_2) \mid ((g_i (\theta_n\times \id_{Z_2}) g_j)(x),x) \in \mathscr T \} )\\
&\leq \sum_{k=1}^N \nu(\{ x \in g_j^{-1} (X_k \times Z_2)\mid ((g_i (\Omega_k \times \id_{Z_2})g_j)(x),x) \in \mathscr T \})\\
&=\sum_{k=1}^N  \varphi_{\mathscr T}(g_i (\Omega_k\times \id_{Z_2}) g_j)\leq \sum_{k=1}^N \frac{1}{n N} = \frac{1}{n}.
\end{align*}
We derive in a similar way using \eqref{erg2} that
for all $1\leq i,j\leq n$, we have
\begin{equation}\label{erg3}
|\langle \kappa^{Z_1}_{\theta_n}(f_i),f_j \rangle|\leq \frac{1}{n}.  
\end{equation}

We can now define $\mathscr S_n$ as the subequivalence relation of $\mathscr S$ generated by $\mathscr S_{n-1}$ and $\theta_n$. Since $\theta_n |_{X_k}$ commutes with $\omega_k$ for any $1\leq k\leq N$, one can deduce that
$\mathscr S_n$ is bounded. By letting $\mathscr S_0=\bigcup_{k\ge 1}\mathscr S_k$, we derive that $\mathscr S_0$ is hyperfinite and that $\mathscr S_0 \times \id_{Z_2}\npreceq_{\mathscr R}  \mathscr T$. By \eqref{erg3}, the Koopman representation associated with $[\mathscr S_0]$ is weakly mixing. This implies by Facts \ref{fact-ergodic}(i) that $\mathscr S_0$ is ergodic.
\end{proof}

The notion of strong primeness for type ${\rm II_1}$ factors was introduced by Isono \cite{Is16}. We adapt this notion to the setting of ergodic equivalence relations.

\begin{defn}\label{defn-strong-primeness}
Let $\mathscr R$ be an ergodic equivalence relation. We say that $\mathscr R$ is {\em strongly prime} when for any ergodic equivalence relations $\mathscr S$, $\mathscr T_1$, $\mathscr T_2$ for which $\mathscr R \times \mathscr S \cong \mathscr T_1 \times \mathscr T_2$, there exists $i \in \{1, 2\}$ such that we have $\mathscr T_i \preceq_{\mathscr R \times \mathscr S} \mathscr S$ as subequivalence relations. 
\end{defn}

\subsection{Algebraic groups}

Let $k$ be a local field, that is, a nondiscrete locally compact field. Throughout, we assume that the characteristic of $k$ is zero. Then $k$ is $\R$, $\C$ or a finite extension of $\mathbb Q_p$ for some prime $p$. In this paper, by a $k$-group $\mathbf G$, we simply mean a linear algebraic group defined over $k$. We denote by $\mathbf G^0$ the Zariski connected component of the identity element $e \in \mathbf G$. We say that $\mathbf G$ is Zariski connected if $\mathbf G = \mathbf G^0$. We say that $\mathbf G$ is almost $k$-simple if $\mathbf G$ is not abelian and the only proper normal $k$-closed subgroups are finite. In that case, since $\mathbf G$ is connected, finite normal subgroups are contained in $\mathscr Z(\mathbf G)$. We say that $\mathbf G$ is $k$-simple if $\mathbf G$ is not abelian and the only normal $k$-closed subgroups are $\{e\}$ and $\mathbf G$. If $\mathbf G$ is (almost) $k$-simple, then we moreover say that $\mathbf G$ is $k$-isotropic if $\rk_k(\mathbf G) \geq 1$. We denote by $G = \mathbf G(k)$ the locally compact second countable group of its $k$-points. Then $G$ is a $k$-analytic Lie group and we naturally have $\Lie(\mathbf G)(k) = \Lie(G)$ as $k$-Lie algebras.

By a $k$-$\mathbf G$-variety $\mathbf V$, we simply mean an algebraic variety defined over $k$ that is endowed with an algebraic action $\mathbf G \curvearrowright \mathbf V$ so that the action map $\mathbf G \times \mathbf V \to \mathbf V : (g, v) \mapsto gv$ is a $k$-morphism. We denote by $V = \mathbf V(k)$ the Hausdorff locally compact second countable topological space of its $k$-points. If $\mathbf V$ is smooth, then $V$ is a $k$-analytic manifold. In the case $\mathbf V = \mathbf G/\mathbf H$, where $\mathbf H < \mathbf G$ is a $k$-subgroup, we denote by $\pi : \mathbf G \to \mathbf G/\mathbf H$ the $\mathbf G$-equivariant $k$-morphism such that $\pi(e) = \mathbf H$. Then  $v = \pi(e) \in (\mathbf G/\mathbf H)(k)$ and $H = \mathbf H(k) = \Stab_G(v)$.

\begin{facts}\label{fact-algebraic}
Keep the same notation as above. We will be using the following useful properties. 
\begin{itemize}
\item [$(\rm i)$] The continuous action $G \curvearrowright V$ has locally closed orbits for the $k$-analytic topology and so the quotient space $$G\backslash V = \left \{ Gv \mid v \in V\right \}$$ is a standard Borel space. 

\item [$(\rm ii)$] The continuous action $G \curvearrowright \Prob(V)$ has locally closed orbits and so the quotient space $$G\backslash \Prob(V) = \left \{ G\mu \mid \mu \in \Prob(V)\right \}$$ is a standard Borel space. 

\item [$(\rm iii)$] Assume that $\mathbf V = \mathbf G/\mathbf H$, where $\mathbf H < \mathbf G$ is a $k$-subgroup. Set $v = \pi(e) \in (\mathbf G/\mathbf H)(k)$ and $H = \mathbf H(k) = \Stab_G(v)$. The orbit map $G/H \to Gv : gH \to gv$ is a homeomorphism, we may identify the $G$-orbit $Gv$ with $G/H$ and we may regard $G/H \subset (\mathbf G/\mathbf H)(k)$ as a closed and open subset. 
\end{itemize}
\end{facts}

\begin{proof}
$(\rm i)$ This follows from \cite[Theorem 2.1.14, Proposition 3.1.3]{Zi84}.

$(\rm ii)$ This follows from \cite[Theorem 2.1.14, Proposition 3.2.6]{Zi84} and more generally \cite[Theorem 1.7]{BDL14}.

$(\rm iii)$  This follows from \cite[Proposition I.2.1.4]{Ma91}.
\end{proof}

\subsection{Algebraic representations of equivalence relations}

Let $\mathscr R$ be an equivalence relation on a standard probability space $(X, \nu)$, $H$ a locally compact second countable group and $\alpha : \mathscr R \to H$ a measurable $1$-cocycle. Let $Z$ be a standard Borel space endowed with a Borel action $H \curvearrowright Z$ and $f : X \to Z$ a Borel map. We say that $f : X \to Z$ is $(\mathscr R, \alpha)$-{\em equivariant} if for $m$-almost every $(x, y) \in \mathscr R$, we have $f(x) = \alpha(x, y)f(y)$. We say that $f : X \to Z$ is $\mathscr R$-{\em invariant} if for $m$-almost every $(x, y) \in \mathscr R$, we have $f(x) = f(y)$.

The following dichotomy theorem generalizes Zimmer's result \cite[Theorem 9.2.3]{Zi84} (see also \cite[Theorem 6.1]{BDL14} for the case of measurable $1$-cocycles associated with nonsingular group actions). For the sake of completeness, we give a short proof.

\begin{thm}\label{thm-non-trivial-gate}
Let $k$ be a local field of characteristic zero and $\mathbf G$ a $k$-isotropic almost $k$-simple algebraic $k$-group. Let $\mathscr R$ be an amenable ergodic equivalence relation on $(X, \nu)$ and $\alpha : \mathscr R \to G$ a measurable $1$-cocycle. Then at least one of the following assertions holds:
\begin{itemize}
\item [$(\rm i)$] There exist a proper $k$-subgroup $\mathbf H < \mathbf G$ and an $(\mathscr R, \alpha)$-equivariant measurable map $\psi : X \to G/H$, where we regard $G/H \subset (\mathbf G/\mathbf H)(k)$ with $H = \mathbf H(k)$.

\item [$(\rm ii)$] There exist a compact subgroup $L < G$ and an $(\mathscr R, \alpha)$-equivariant measurable map $f : X \to G/L$.
\end{itemize}
\end{thm}
Observe that in the case $k = \R$, assertion $(\rm i)$ always holds.

\begin{proof}
Choose a minimal parabolic $k$-subgroup $\mathbf P < \mathbf G$ and denote by $\mathbf V = \mathbf G/\mathbf P$ the homogeneous $k$-$\mathbf G$-variety. Write $G = \mathbf G(k)$, $P = \mathbf P(k)$ and $V = (\mathbf G/\mathbf P)(k) = \mathbf G(k)/\mathbf P(k)$ (see \cite[Proposition 20.5]{Bo91}). Since $V$ is compact and since $\mathscr R$ is amenable, there exists an $(\mathscr R, \alpha)$-equivariant measurable map $\beta : X \to \Prob(V)$ (see \cite[Proposition 4.3.9]{Zi84}). By Facts \ref{fact-algebraic}$(\rm ii)$, the continuous action $G \curvearrowright \Prob(V)$ has locally closed orbits and so the quotient space $G\backslash \Prob(V)$ is a standard Borel space. Denote by $p : \Prob(V) \to G\backslash\Prob(V)$ the quotient Borel map. Then the measurable map $p \circ \beta : X \to G\backslash\Prob(V)$ is $\mathscr R$-invariant. Since $\mathscr R$ is ergodic, $p \circ \beta : X \to G\backslash\Prob(V)$ is $\nu$-almost everywhere constant and so there exists $\mu \in \Prob(V)$ such that $\beta(X)$ is essentially contained in $G \mu$. Set $L = \Stab_G(\mu) < G$. By \cite[Theorem 2.1.14]{Zi84}, the orbit map $G/L \to G\mu : gL \to g\mu$ is a homeomorphism and so we may regard $\beta : X \to G/L$ as an $(\mathscr R, \alpha)$-equivariant measurable map.

Denote by $\mathbf H$ the Zariski closure of $L$ in $\mathbf G$. Then $\mathbf H < \mathbf G$ is a $k$-subgroup and we set $H = \mathbf H(k)$. By \cite[Proposition 1.9]{BDL14}, there exists a $k$-subgroup $\mathbf H_0 < \mathbf H < \mathbf G$ such that $\mathbf H_0 \lhd \mathbf H$ is normal, the image of $L$ is precompact in $(\mathbf H/\mathbf H_0)(k)$ and $\mu$ is supported on $\mathbf V^{\mathbf H_0} \cap V$. Since $\mathbf V^{\mathbf G} = \emptyset$, we have $\mathbf H_0 \neq \mathbf G$. 

Assume that $\mathbf H_0$ is infinite. Since $\mathbf H_0 \lhd \mathbf H$, $\mathbf H_0 \neq \mathbf G$ and $\mathbf G$ is almost $k$-simple, we have $\mathbf H \neq \mathbf G$. Since $L < H$, we may consider the $G$-equivariant measurable factor map $q : G/L \to G/H$. Then the measurable map $f = q \circ \beta : X \to G/H$ is $(\mathscr R, \alpha)$-equivariant. 

Assume that $\mathbf H_0$ is finite. Set $H_0 = \mathbf H_0(k)$. The image of $L$ in $(\mathbf H/\mathbf H_0)(k)$ is contained in $H/H_0 \subset (\mathbf H/\mathbf H_0)(k)$, which is closed in $(\mathbf H/\mathbf H_0)(k)$. Then the image of $L$ in $H/H_0$ is precompact. Since $H_0 < H$ is finite, this further implies that $L < H$ is compact.
\end{proof}

We use \cite[Theorem 6.1]{BDL14} in combination with \cite[Theorem 14]{TD14} to show that Zariski dense discrete subgroups are not inner amenable.

\begin{prop}\label{prop-not-inner-amenable}
    Let $k$ be a local field of characteristic zero and $\mathbf G$ a $k$-isotropic almost $k$-simple algebraic $k$-group. Let $\Gamma < \mathbf G(k)$ be a Zariski dense discrete subgroup. Then $\Gamma$ is not inner amenable.
\end{prop}

\begin{proof}
By contradiction, assume that $\Gamma < \mathbf G(k)$ is inner amenable. By \cite[Theorem 14]{TD14}, there exists a short exact sequence $1 \to N \to \Gamma \to K \to 1$, where $K$ is amenable and either
\begin{itemize}
\item [$(\rm i)$] $\mathscr Z(N)$ is infinite, or
\item [$(\rm ii)$] $N = LM$, where $L$ and $M$ are commuting normal subgroups of $\Gamma$ such that $M$ is infinite and amenable, and $L \cap M$ is finite.
\end{itemize}

In case $(\rm i)$, since $\mathscr Z(N) \lhd \Gamma$ is an infinite normal subgroup, it follows that $\mathscr Z(N)$ is Zariski dense in $\mathbf G$. This would imply that $\mathbf G$ is abelian, which is absurd.  

In case $(\rm ii)$, since $M \lhd \Gamma$ is an infinite normal subgroup, it follows that $M$ is Zariski dense in $\mathbf G$. Since $M$ is amenable, applying \cite[Theorem 6.1]{BDL14} to the trivial action $M \curvearrowright \{\bullet\}$ which is amenable, at least one of the following assertions holds:
\begin{itemize}

\item There exist a proper $k$-subgroup $\mathbf H < \mathbf G$ and an $M$-equivariant measurable map $\psi : \{\bullet\} \to G/H$, where we regard $G/H \subset (\mathbf G/\mathbf H)(k)$ with $H = \mathbf H(k)$. Then there exists $g \in G$ such that $M < g H g^{-1}$ and so the Zariski closure of $M$ is a proper $k$-subgroup of $\mathbf G$, which is a contradiction.

\item There exist a compact subgroup $L < G$ and an $M$-equivariant measurable map $f : \{\bullet\} \to G/L$. Then there exists $g \in G$ such that $M < g L g^{-1}$ and so $M$ is compact, which is a contradiction. 
\end{itemize} 
This shows that $\Gamma$ is not inner amenable.
\end{proof}

\begin{rem}\label{remark-icc}
 Let $k$ be a local field of characteristic zero and $\mathbf G$ a $k$-isotropic $k$-simple algebraic $k$-group. Let $\Gamma < \mathbf G(k)$ be a Zariski dense discrete subgroup. Then $\Gamma$ is icc meaning that $\Gamma$ has infinite conjugacy classes. Indeed, let $\gamma \in \Gamma$ be such that its conjugacy class $C(\gamma) = \left \{ h \gamma h^{-1} \mid h \in \Gamma\right \}$ is finite. Then $\mathscr Z_\Gamma(\gamma) < \Gamma$ has finite index and so $\mathscr Z_\Gamma(\gamma)$ is Zariski dense in $\mathbf G$ by Borel's density theorem \cite{Bo60}. This implies that $\gamma \in \mathscr Z(\mathbf G)$. Since $\mathbf G$ is $k$-simple, we have $\gamma = e$.
\end{rem}

\subsection{Chabauty topology and Wijsman topology}

Let $G$ be a locally compact second countable group. Consider the space $\Sub(G)$ of all closed subgroups of $G$. Endowed with the Chabauty topology, $\Sub(G)$ is a compact metrizable space and the conjugation action $G \curvearrowright \Sub(G)$ is continuous (see \cite{Ch50}). Recall that for any sequence $(H_n)_{n \in \N}$ in $\Sub(G)$ and any element $H \in \Sub(G)$, we have $H_n \to H$ in $\Sub(G)$ if and only if the following two conditions hold:
\begin{itemize}
\item [$(\rm i)$] For every $x \in H$, there exists a sequence $(x_n)_{n \in \N}$ in $G$ such that $x_n \in H_n$ for every $n \in \N$ and $x_n \to x$ in $G$. 
\item [$(\rm ii)$] For every increasing sequence $(n_j)_{j \in \N}$ in $\N$, every sequence $(x_{j})_{j \in \N}$ in $G$ such that $x_{j} \in H_{n_j}$ for every $j \in \N$, and every element $x \in G$, if $x_{j} \to x$ in $G$, then $x \in H$.
\end{itemize}
We refer the reader to \cite[Section E.1]{BP92} for further details. Then for every $K \in \Sub(G)$, the subset
\begin{equation}\label{eq-closed-bis}
\left \{H \in \Sub(G) \mid H < K \right \} \subset \Sub(G) \quad \text{is closed}.
\end{equation}

Let $\mathscr G$ be a Polish group and choose a compatible complete metric $d : \mathscr G \times \mathscr G \to \R_+$. Denote by $\CL(\mathscr G)$ the space of all nonempty closed subsets of $\mathscr G$. For all $x\in \mathscr G$, $A \in \CL(\mathscr G)$, set $d(x, A) = \inf \left \{d(x, a) \mid a \in A\right \}$. Following \cite{Be90}, the Wijsman topology $\tau_{W(d)}$ on $\CL(\mathscr G)$ is defined as the weakest topology on $\CL(\mathscr G)$ that makes the maps $\CL(\mathscr G) \to \R_+ : A \mapsto d(x, A)$ continuous for all $x \in \mathscr G$. As explained in \cite{Be90}, the Borel structure induced by the Wijsman topology $\tau_{W(d)}$ coincides with the Effros--Borel structure on $\CL(\mathscr G)$. By \cite[Theorem 4.3]{Be90}, $(\CL(\mathscr G), \tau_{W(d)})$ is a Polish space. Consider the space $\Sub(\mathscr G)$ of all closed subgroups of $\mathscr G$. Then $\Sub(\mathscr G) \subset \CL(\mathscr G)$ is closed with respect to $\tau_{W(d)}$ and so $(\Sub(\mathscr G), \tau_{W(d)})$ is a Polish space. Observe that the subset
\begin{equation}\label{eq-closed}
\left \{ (h, \mathscr H) \in \mathscr G \times \Sub(\mathscr G) \mid h \in \mathscr H \right \} \subset \mathscr G \times \Sub(\mathscr G) \quad \text{is closed}.
\end{equation}
When $\mathscr G = G$ is a locally compact second countable group and $d_G : G \times G \to \R_+$ is a left invariant compatible proper complete metric (see \cite{St73}), the Wijsman topology $\tau_{W(d_G)}$ and the Chabauty topology coincide on $\Sub(G)$ (see \cite[Section 5]{BLLN89}).

Let now $G$ be a locally compact second countable group. Choose a left invariant compatible proper complete metric $d_{G} : G \times G \to \R_+$ (see \cite{St73}). Then $d = \min(d_G, 1) : G \times G \to \R_+$ is a left invariant compatible complete metric such that $d \leq 1$. Let $(X, \nu)$ be a standard probability space. Denote by $\mathscr G = \rL^0(X, \nu, G)$ the space of all $\nu$-equivalence classes of measurable maps $f : X \to G$. Endowed with the topology of convergence in measure, $\mathscr G$ is a Polish group (see e.g.\! \cite[Section 19]{Ke10}). More precisely, define the metric $d_{\mathscr G} : \mathscr G \times \mathscr G \to \R_+$ by the formula
$$\forall f_1, f_2 \in \mathscr G, \quad d_{\mathscr G}(f_1, f_2) = \int_X d(f_1(x), f_2(x)) \, {\rm d}\nu(x).$$
Then $d_{\mathscr G}$ is a left invariant compatible complete metric on $\mathscr G$. For every $H \in \Sub(G)$, we have that $\mathscr H = \rL^0(X, \nu, H) \in \Sub(\mathscr G)$. The following technical result will be very useful in the proof of Theorem \ref{thm-product-algebraic}.

\begin{prop}\label{prop-borel}
Keep the same notation as above. Then the map 
$$\Psi : \Sub(G) \to \Sub(\mathscr G) : H \mapsto \rL^0(X, \nu, H)$$ 
is Borel.
\end{prop}

\begin{proof}
Consider the Wijsman topology $\tau = \tau_{W(d_{\mathscr G})}$ on $\Sub(\mathscr G)$. Fix a countable dense subset $D \subset \mathscr G$. For every $f \in D$ and all $\alpha, \beta \in \mathbb Q \cap [0, 1]$ such that $\alpha < \beta$, define the open subsets 
\begin{align*}
O(f, \alpha)^+ &= \left \{\mathscr K \in \Sub(\mathscr G) \mid d_{\mathscr G}(f, \mathscr K) > \alpha \right \} \\
O(f, \beta)^- &= \left \{\mathscr K \in \Sub(\mathscr G) \mid d_{\mathscr G}(f, \mathscr K) < \beta \right \} \\
O(f, \alpha, \beta) &= O(f, \alpha)^+ \cap O(f, \beta)^-.
\end{align*}
Then the countable set $\left \{ O(f, \alpha, \beta) \mid f \in D, \alpha, \beta \in \mathbb Q \cap [0, 1],\alpha < \beta \right \}$ is a subbase for the topology $\tau$ on $\Sub(\mathscr G)$. Thus, in order to show that the map $\Psi : \Sub(G) \to \Sub(\mathscr G)$ is Borel, it suffices to show that $\Psi^{-1}(O(f, \alpha)^+)$ and $\Psi^{-1}(O(f, \alpha)^-)$ are Borel subsets of $\Sub(G)$ for every $f \in D$ and every $\alpha \in \mathbb Q \cap [0, 1]$.

\begin{claim}\label{claim-lusin}
Let $(H_n)_{n \in \N}$ be a sequence in $\Sub(G)$ and $H \in \Sub(G)$ such that $H_n \to H$ in $\Sub(G)$. Set $\mathscr H_n = \rL^0(X, \nu, H_n)$ for every $n \in \N$ and $\mathscr H = \rL^0(X, \nu, H)$. Then for every $f \in \mathscr H$, there exists a sequence $(f_n)_{n \in \N}$ in $\mathscr G$ such that $f_n \in \mathscr H_n$ for every $n \in \N$ and $\lim_n d_{\mathscr G}(f, f_n) = 0$.
\end{claim}

Indeed, let $f \in \mathscr H$. Without loss of generality, we may assume that $X$ is a compact metrizable space and $\nu \in \Prob(X)$ is a Borel probability measure. Regard $f : X \to H$ as a measurable function. By Lusin's theorem, for every $n \geq 1$, there exists a closed subset $X_n \subset X$ such that $\nu(X_n) \geq 1 - 1/n$ and $f|_{X_n} : X_n \to H$ is continuous. By compactness, there exist $r \geq 1$ and $h_1, \dots, h_r \in H$ such that $X_n \subset \bigcup_{j = 1}^r f^{-1}(B_{d}(h_j, 1/n))$, where $B_d(h, \varepsilon) \subset G$ denotes the open ball in $G$ of center $h \in G$ and radius $\varepsilon > 0$. Set $Y_1 = X_n \cap f^{-1}(B_{d}(h_1, 1/n))$ and $Y_j = X_n \cap f^{-1}(B_{d}(h_j, 1/n)) \setminus \bigcup_{i = 1}^{j - 1} Y_i$ for every $2 \leq j \leq r$. Then $(Y_j)_{1 \leq j \leq r}$ is a measurable partition of $X_n$. Since $H_n \to H$ in $\Sub(G)$, for every $1 \leq j \leq r$, we may find $h_j^n \in B_{d}(h_j, 1/n) \cap H_n$. Define the measurable map $f_n : X \to H_n$ by $f_n|_{X \setminus X_n} = e \mathbf 1_{X \setminus X_n}$ and $f_n|_{Y_j} = h_j^n \mathbf 1_{Y_j}$ for every $1 \leq j \leq r$. Then we have
\begin{align*}
\limsup_n d_{\mathscr G}(f, f_n) &= \limsup_n \int_X d(f(x), f_n(x)) \,{\rm d}\nu(x) \\
&\leq \limsup_n \left(\nu(X \setminus X_n) + \frac2n \nu(X_n) \right) =0.
\end{align*}
This finishes the proof of Claim \ref{claim-lusin}.

\begin{claim}\label{claim-closed}
Let $(H_n)_{n \in \N}$ be a sequence in $\Sub(G)$ and $H \in \Sub(G)$ such that $H_n \to H$ in $\Sub(G)$. Set $\mathscr H_n = \rL^0(X, \nu, H_n)$ for every $n \in \N$ and $\mathscr H = \rL^0(X, \nu, H)$. Then for every $f \in \mathscr G$, we have
$$\limsup_n d_{\mathscr G}(f, \mathscr H_n) \leq d_{\mathscr G}(f, \mathscr H).$$
\end{claim}

Indeed, set $\alpha = d_{\mathscr G}(f, \mathscr H)$ and $\beta = \limsup_n d_{\mathscr G}(f, \mathscr H_n)$. Choose an increasing sequence $(n_j)_{j \in \N}$ in $\N^*$ such that $\beta = \lim_j d_{\mathscr G}(f, \mathscr H_{n_j})$. Choose a sequence $(f_n)_{n \in \N}$ in $\mathscr H$ such that $\alpha \leq d_{\mathscr G}(f, f_n) < \alpha + 1/n$ for every $n \geq 1$. By Claim \ref{claim-lusin}, we may choose a sequence $(h_n)_{n \geq 1}$ in $\mathscr G$ such that $h_n \in \mathscr H_n$ and $d_{\mathscr G}(f_n, h_n) \leq 1/n$ for every $n \geq 1$. In particular, for every $j \in \N$, we have 
$$d_{\mathscr G}(f, \mathscr H_{n_j}) \leq d_{\mathscr G}(f, h_{n_j}) \leq d_{\mathscr G}(f, f_{n_j}) + d_{\mathscr G}(f_{n_j}, h_{n_j}) \leq \alpha + \frac{2}{n_j}.$$
This implies that $\beta \leq \alpha$ and finishes the proof of Claim \ref{claim-closed}.

Let $f \in D$ and $\alpha \in \mathbb Q \cap [0, 1]$. Claim \ref{claim-closed} implies that $\Psi^{-1}(O(f, \alpha)^-) \subset \Sub(G)$ is open. Observe that 
$$O(f, \alpha)^+ = \bigcup_{j = 1}^\infty  \left \{\mathscr K \in \Sub(\mathscr G) \mid d_{\mathscr G}(f, \mathscr K) \geq \alpha + 1/j \right \}.$$
Claim \ref{claim-closed} implies that $\Psi^{-1}(O(f, \alpha)^+) \subset \Sub(G)$ is a countable union of closed subsets and so $\Psi^{-1}(O(f, \alpha)^+) \subset \Sub(G)$ is Borel. Therefore, we have showed that the map $\Psi : \Sub(G) \to \Sub(\mathscr G) $ is Borel.
\end{proof}

\section{Proofs of the main results}

We will frequently use the following version of the Jankov--von Neumann selection theorem.

\begin{thm}\label{thm-selection}
Let $(X, \nu)$ be a standard probability space, $Y$ a Polish space and $B \subset X \times Y$ a Borel subset such that $\pi_X(B) = X$, where $\pi_X : X \times Y \to X$ is the projection map. Then there exists a Borel map $f : X \to Y$ such that $(x, f(x)) \in B$ for $\nu$-almost every $x \in X$.
\end{thm}

\begin{proof}
By the Jankov--von Neumann selection theorem, there exists a $\sigma(\Sigma_1^1)$-measurable map $f : X \to Y$ such that $(x, f(x))$ for every $x \in X$ (see \cite[Theorem 18.1]{Ke95}). We then use two well-known facts. Firstly, any $\sigma(\Sigma_1^1)$-measurable set is Lebesgue measurable (see \cite[Theorem 29.7]{Ke95}). Secondly, any Lebesgue measurable map $X \to Y$ coincides $\nu$-almost everywhere with a Borel map. Therefore, upon modifying $f : X \to Y$ on a $\nu$-null Borel subset, we may assume that the map $f : X \to Y$ is Borel and satisfies $(x, f(x)) \in B$ for $\nu$-almost every $x \in X$.
\end{proof}

Next, let $(X, \nu)$ and $(Y, \eta)$ be two standard probability spaces. Recall that the spaces $\rL^0(X \times Y, \nu \otimes \eta, \C)$, $\rL^0(Y, \eta, \C)$ and $\rL^0(X, \nu, \rL^0(Y, \eta, \C))$ are Polish spaces when endowed with the topology of convergence in measure. Moreover, the well-defined map 
$$\iota : \rL^0(X \times Y, \nu \otimes \eta, \mathbb C) \to \rL^0(X, \nu, \rL^0(Y, \eta, \C)) : f \mapsto (x \mapsto f(x, \cdot \,))$$
is isometric (hence injective) and surjective. Therefore, we may and will identify $\rL^0(X \times Y, \nu \otimes \eta, \mathbb C)$ and $\rL^0(X, \nu, \rL^0(Y, \eta, \C))$ (see e.g.\! \cite[Theorem 1]{Mo75}).

\subsection{Zimmer's method for product equivalence relations}

Let $k$ be a local field of characteristic zero and $\mathbf G$ a $k$-group. Set $G = \mathbf G(k)$. 

Let $\mathscr U$ be an ergodic equivalence relation on $(U, \nu)$ and $\mathscr V$ an ergodic equivalence relation on $(V, \eta)$. Set $\mathscr W = \mathscr U \times \mathscr V$ and $(W, \zeta) = (U \times V, \nu \otimes \eta)$. Let $\alpha : \mathscr W \to G$ be a Borel $1$-cocycle. Following \cite{Zi81}, we define the set $\mathfrak A(\alpha)$ that consists of all pairs $(\varphi, \mathfrak H)$, where $\varphi : W \to G$ is a Borel map, $\mathfrak H : U \to \Sub(G) : u \mapsto H_u$ is a Borel map such that for $\nu$-almost every $u \in U$, there exists a $k$-subgroup $\mathbf H_u < \mathbf G$ such that $\mathbf H_u(k) = H_u$ and the Borel map $V \to G/H_u : v \mapsto \varphi(u, v)H_u$ is $(\mathscr V, \alpha_u)$-equivariant, where $\alpha_u : \mathscr V \to G : (v_1, v_2) \mapsto \alpha((u, v_1), (u, v_2))$. Observe that the trivial pair $(\varphi, \mathfrak H)$ defined by $\varphi : W \to G : (u, v) \mapsto e$ and $\mathfrak H : U \to \Sub(G) : u \mapsto G$ belongs to $\mathfrak A(\alpha)$. Thus, $\mathfrak A(\alpha)$ is not the empty set. We note that the set $\mathfrak A(\alpha)$ only depends on the restriction of $\alpha$ to $\id_U \times \mathscr V$.

As explained in \cite{Zi81}, the set $\mathfrak A(\alpha)$ is stable under the following cutting and pasting procedures. To do this, we will exploit the fact that the Borel $1$-cocycle $\alpha$ is defined on $\mathscr U \times \mathscr V = \mathscr W$. Firstly, let $(\varphi, \mathfrak H) \in \mathfrak A(\alpha)$ and $\lambda \in [\mathscr U]$. Define the pair $(\psi, \mathfrak J) = \lambda_\ast(\varphi, \mathfrak H)$ by $\psi : W \to G : (u, v) \mapsto \alpha((u, v), (\lambda u, v))\varphi(\lambda u, v)$ and $\mathfrak J : U \to \Sub(G) : u \mapsto H_{\lambda u}$. Then a straightforward computation shows that $(\psi, \mathfrak J) \in \mathfrak A(\alpha)$. Secondly, let $(\varphi^i, \mathfrak H^i) \in \mathfrak A(\alpha)$ for every $i \in \{1, 2\}$. Let $U = U_1 \sqcup U_2$ be a measurable partition. Define the pair $(\varphi, \mathfrak H)$ by $\varphi|_{U_i \times V} = \varphi^i|_{U_i \times V}$ and $\mathfrak H|_{U_i \times V} = \mathfrak H^i|_{U_i \times V}$ for every $i \in \{1, 2\}$. Then it is plain to see that $(\varphi, \mathfrak H) \in \mathfrak A(\alpha)$. We record the following useful result.

\begin{lem}\label{lem-initial}
Assume that the probability space $(U, \nu)$ is nonatomic. There exists a pair $(\varphi, \mathfrak H) \in \mathfrak A(\alpha)$ such that the $\nu$-almost everywhere defined map $U \mapsto \N : u  \mapsto \dim(\mathbf H_u)$ is $\nu$-almost everywhere constant and for any other pair $(\psi, \mathfrak J) \in \mathfrak A(\alpha)$, we have $\dim(\mathbf J_u) \geq \dim(\mathbf H_u)$ for $\nu$-almost every $u \in U$.
\end{lem}

\begin{proof}
Consider the nonempty finite subset $\mathscr F \subset \N$ that consists of all integers $j \in \N$ for which there exists $(\psi, \mathfrak J) \in \mathfrak A(\alpha)$ such that $\nu(\{ u \in U \mid \dim(\mathbf J_u) = j\}) > 0$. Set $d = \min(\mathscr F)$ and choose a pair $(\psi, \mathfrak J) \in \mathfrak A(\alpha)$ such that $\nu(\{ u \in U \mid \dim(\mathbf J_u) = d \}) > 0$. Choose $n \geq 1$ and a measurable subset $U_1 \subset \{ u \in U \mid \dim(\mathbf J_u) = d\}$ such that $\nu(U_1) = \frac1n$. Choose measurable subsets $U_2, \dots, U_n \subset X$ so that $(U_i)_{1 \leq i \leq n}$ forms a measurable partition of $X$. Since $\mathscr U$ is ergodic on $(U, \nu)$, for every $1 \leq i \leq n$, we may choose $\lambda_i \in [\mathscr U]$ such that $\lambda_i(U_i) = U_1$. Then for every $1 \leq i \leq n$, we have $(\psi^i, \mathfrak J^i) = {\lambda_i}_\ast(\psi, \mathfrak J) \in \mathfrak A(\alpha)$. Define the pair $(\varphi, \mathfrak H) \in \mathfrak A(\alpha)$ by $\varphi|_{U_i \times V} = \psi^i|_{U_i \times V}$ and $\mathfrak H|_{U_i \times V} = \mathfrak J^i|_{U_i \times V}$ for every $1 \leq i \leq n$. Then we have $d = \dim(\mathbf H_u)$ for $\nu$-almost every $u \in U$.
\end{proof}

We then simply say that $(\varphi, \mathfrak H)  \in \mathfrak A(\alpha)$ is a minimal pair and we write $d = \dim(\mathbf H_u)$ for $\nu$-almost every $u \in U$.

The following theorem is an extension of \cite[Lemma 2.1]{Zi81} to the case of algebraic groups defined over a local field $k$ of characteristic zero. For the sake of completeness, we give a proof following the one of \cite[Lemma 2.1]{Zi81} and we add a few details regarding the measurability of certain maps.

\begin{thm}\label{thm-product-algebraic}
Assume that the probability space $(U, \nu)$ is nonatomic. Keep the same notation as above. The following assertions hold:

\begin{itemize}
\item [$(\rm i)$] Let $(\varphi, \mathfrak H) \in \mathfrak A(\alpha)$ and $(\psi, \mathfrak J) \in \mathfrak A(\alpha)$ be two minimal pairs. Then there exists a measurable map $U \to G : u \mapsto b_u$ such that $\mathbf H_u^0 = b_u \mathbf J_u^0 b_u^{-1}$ for $\nu$-almost every $u \in U$.

\item [$(\rm ii)$] Let $(\varphi, \mathfrak H) \in \mathfrak A(\alpha)$ be a minimal pair. Then there exists a Zariski connected $k$-subgroup $\mathbf H < \mathbf G$ and a measurable map $U \to G : u \mapsto g_u$ such that $\mathbf H_u^0 = g_u \mathbf H g_u^{-1}$ for $\nu$-almost every $u \in U$.

\item [$(\rm iii)$] Let $(\varphi, \mathfrak H) \in \mathfrak A(\alpha)$ and $(\psi, \mathfrak J) \in \mathfrak A(\alpha)$ be two minimal pairs for which there exists a Zariski connected $k$-subgroup $\mathbf H < \mathbf G$ such that $\mathbf H_u^0 = \mathbf H = \mathbf J_u^0$ for $\nu$-almost every $u \in U$. Set $\mathbf L = \mathscr N_{\mathbf G}(\mathbf H)$, $H = \mathbf H(k)$, $L = \mathbf L(k)$ and consider the factor map $p : G \to G/L$. Then we have $p \circ \varphi = p \circ \psi$ $\zeta$-almost everywhere.
\end{itemize}
\end{thm}

\begin{proof}
$(\rm i)$ After discarding a $\nu$-null Borel subset, we may assume that the Borel maps $\mathfrak H : U \to \Sub(G) : u \mapsto H_u$ and $\mathfrak J : U \to \Sub(G) : u \mapsto J_u$ satisfy that for every $u \in U$, there exist a $k$-subgroup $\mathbf H_u < \mathbf G$ such that $\mathbf H_u(k) = H_u$ and a $k$-subgroup $\mathbf J_u < \mathbf G$ such that $\mathbf J_u(k) = J_u$, and the Borel maps $V \to G/H_u : v \mapsto \varphi(u, v)H_u$ and $V \to G/J_u : v \mapsto \psi(u, v)J_u$ are $(\mathscr V, \alpha_u)$-equivariant. Denote by $\mathscr G = \rL^0(V, \eta, G)$ the space of all $\eta$-equivalence classes of measurable functions $f : V \to G$. Endowed with the topology of convergence in measure, $\mathscr G$ is a Polish group. 

In this paragraph, we fix $u \in U$. Consider the $k$-$\mathbf G$-variety $\mathbf M_u = \mathbf G/\mathbf H_u \times \mathbf G/\mathbf J_u$, where $\mathbf G \curvearrowright \mathbf M_u$ acts diagonally. Set $M_u = \mathbf M_u(k)$ and regard $N_u = G/H_u \times G/J_u \subset M_u$ as a closed subset. Define the Borel map 
$$\theta_u : V \to N_u : v \mapsto (\varphi(u, v)H_u, \psi(u, v)J_u).$$
Since $\mathbf M_u$ is an algebraic $k$-$\mathbf G$-variety, the continuous action $G \curvearrowright M_u$ has locally closed orbits (see Facts \ref{fact-algebraic}$(\rm i)$). Thus, the action $G \curvearrowright N_u$ has locally closed orbits and so the quotient space $G\backslash N_u$ is a standard Borel space. Since the measurable map 
$$V \to G\backslash N_u : v \mapsto G\cdot \theta_u(v)$$ is $\mathscr V$-invariant and since $\mathscr V$ is ergodic, $V \to G\backslash N_u : v \mapsto G\cdot \theta_u(v)$ is $\eta$-almost everywhere constant. We may choose $b \in G$ such that $\theta_u(v) \in G \cdot (H_u, b J_u)$ for $\eta$-almost every $v \in V$. Letting $L = \Stab_G(H_u, b J_u) < G$, the map 
$$G/L \to G \cdot (H_u, b J_u) : gL \mapsto g(H_u, bJ_u)$$
 is a homeomorphism. By considering a Borel section $G/L \to G$ of the projection map $G \to G/L$, we may then find $F \in \mathscr G$ such that $\theta_u(v) = (F(v)H_u, F(v)b J_u)$ for $\eta$-almost every $v \in V$. Define $\varphi_u : V \to G : v \mapsto \varphi(u, v)$, $\psi_u : V \to G : v \mapsto \psi(u, v)$ and regard $\varphi_u \in \mathscr G$, $\psi_u \in \mathscr G$. 
 
 Then the maps $U \to \mathscr G : u \mapsto \varphi_u$ and $U \to \mathscr G : u \mapsto \psi_u$ are Borel.  For every $u \in U$, set $\mathscr H_u = \rL^0(V, \eta, H_u)$ and $\mathscr J_u = \rL^0(V, \eta, J_u)$. By Proposition \ref{prop-borel}, the maps $U \to \Sub(\mathscr G) : u \mapsto \mathscr H_u$ and $U \to \Sub(\mathscr G) : u \mapsto \mathscr J_u$ are Borel. Denote by $\pi_U : U \times G \times \mathscr G \to U$ the projection map. Consider the subset
$$B = \left\{ (u, b, F) \in U \times G \times \mathscr G \mid F^{-1}\varphi_u \in \mathscr H_u \; \text{ and } \;  b^{-1}F^{-1}\psi_u \in \mathscr J_u \right\}.$$ 
The reasoning in the previous paragraph shows that for every $u \in U$, there exist $F \in \mathscr G$ and $b \in G$ such that $(\varphi(u, v)H_u, \psi(u, v)J_u) = (F(v)H_u, F(v)b J_u)$ for $\eta$-almost every $v \in V$. This means exactly that $(u, b, F) \in B$. Therefore, we have $\pi_U(B) = U$. Using \eqref{eq-closed}, we have that $B \subset U \times G \times \mathscr G$ is a Borel subset. By the measurable selection theorem (see Theorem \ref{thm-selection}), there exists a Borel map $U \to G \times \mathscr G : u \mapsto (b_u, F_u)$ such that $(u, b_u, F_u) \in B$ for $\nu$-almost every $u \in U$. Thus, we may choose Borel maps $\theta : W \to G$ and $\mathfrak L : U \mapsto \Sub(G) : u \mapsto L_u$ such that $\theta(u, v) = F_u(v)$ for $\zeta$-almost every $(u, v) \in W$ and $L_u = H_u \cap b_u J_u b_u^{-1}$ for $\nu$-almost every $u \in U$. Note that for $\nu$-almost every $u \in U$, letting $\mathbf L_u = \mathbf H_u \cap b_u \mathbf J_u b_u^{-1} < \mathbf G$, which is a $k$-subgroup, we have $L_u = \mathbf L_u(k)$. For $\nu$-almost every $u \in U$, using the $G$-equivariant homeomorphism 
$$G/L_u \to G \cdot (H_u, b_u J_u) : g L_u \mapsto (gH_u, g b_u J_u),$$ we may identify 
\begin{align*}
\theta(u, v)L_u &= (\theta(u, v)H_u, \theta(u, v) b_u J_u) \\
&= (F_u(v) H_u, F_u(v) b_u J_u) \\
&= (\varphi(u, v)H_u, \psi(u, v)J_u).
\end{align*}
for $\zeta$-almost every $(u, v) \in W$. Since for $\nu$-almost every $u \in U$, the Borel maps $V \to G/H_u : v \mapsto \varphi(u, v)H_u$ and $V \to G/J_u : v \mapsto \psi(u, v)J_u$ are $(\mathscr V, \alpha_u)$-equivariant, it follows that the Borel map $V \to G/L_u : v \mapsto \theta(u, v)L_u$ is $(\mathscr V, \alpha_u)$-equivariant. Therefore, we infer that $(\theta, \mathfrak L) \in \mathfrak A(\alpha)$. By minimality of the pairs $(\varphi, \mathfrak H) \in \mathfrak A(\alpha)$ and $(\psi, \mathfrak J) \in \mathfrak A(\alpha)$, we have that $\dim(\mathbf H_u) = \dim(\mathbf H_u \cap b_u \mathbf J_u b_u^{-1}) = \dim(\mathbf J_u)$ for $\nu$-almost every $u \in U$. This further implies that $\mathbf H_u^0 = (\mathbf H_u \cap b_u \mathbf J_u b_u^{-1})^0 = b_u \mathbf J_u^0 b_u^{-1}$ for $\nu$-almost every $u \in U$.

$(\rm ii)$ Consider the adjoint $k$-representation $\Ad : \mathbf G \to \GL(\Lie(\mathbf G))$. Denote by $M = \Gr_d(\Lie(G))$ (resp.\! $\mathbf M = \Gr(\Lie(\mathbf G))$) the $k$-analytic Grassmannian $G$-manifold (resp.\! algebraic $k$-$\mathbf G$-variety) of all $d$-dimensional subspaces of $\Lie(G)$ (resp.\! $\Lie(\mathbf G)$). We naturally have $M = \mathbf M(k)$ as $k$-analytic manifolds. The map $U \to M : u \mapsto \Lie(H_u)$ is measurable (see \cite[Section 3]{NZ00}). Since $\mathbf M$ is an algebraic $k$-$\mathbf G$-variety, the continuous action $G \curvearrowright M$ has locally closed orbits and so the quotient space $G\backslash M$ is a standard Borel space (see Facts \ref{fact-algebraic}$(\rm i)$). Consider the measurable map $\beta : U \to G\backslash M : u \mapsto G \cdot \Lie(H_u)$. Let $\lambda \in [\mathscr U]$. Since $\lambda_\ast (\varphi, \mathfrak H) \in \mathfrak A(\alpha)$ is a minimal pair, item $(\rm i)$ implies that there exists a measurable map $U \to G : u \mapsto b_u$ such that $\mathbf H_u^0 = b_u \mathbf H_{\lambda u}^0 b_u^{-1}$ for $\nu$-almost every $u \in U$. Then we have $\beta(\lambda u) = \beta(u)$ for $\nu$-almost every $u \in U$. This implies that the measurable map $\beta : U \to G\backslash M$ is $\mathscr U$-invariant. Since $\mathscr U$ is ergodic, $\beta : U \to G\backslash M$ is $\nu$-almost everywhere constant and so there exists a point $w \in M$ such that $\Lie(H_u) \in G \cdot w$ for $\nu$-almost every $u \in U$. Then there exists a measurable map $U \to G : u \mapsto g_u$ such that $\Lie(H_u) = g_u \cdot w$ for $\nu$-almost every $u \in U$. By Fubini's theorem, there exists $u_0 \in U$ such that letting $\mathbf H = g_{u_0}^{-1} \mathbf H_{u_0}^0 g_{u_0}$, we have $\Lie(\mathbf H_u^0) = \Lie(g_u \mathbf H g_u^{-1})$ for $\nu$-almost every $u \in U$. Since the characteristic of $k$ is zero, it follows that $\mathbf H_u^0 = g_u \mathbf H g_u^{-1}$ for $\nu$-almost every $u \in U$.

$(\rm iii)$ By the proof of item $(\rm i)$, there exist measurable maps $\theta : W \to G$ and $U \to G : u \mapsto b_u$ such that 
$$(\varphi(u, v)H_u, \psi(u,v) J_u) = (\theta(u, v) H_u, \theta(u, v) b_u J_u)$$
for $\zeta$-almost every $(u, v) \in W$. For $\nu$-almost every $u \in U$, since $\dim(\mathbf H_u) = \dim(\mathbf H_u \cap b_u \mathbf J_u b_u^{-1}) = \dim(\mathbf J_u)$ and $\mathbf H_u^0 = \mathbf H = \mathbf J_u^0$, we have $\mathbf H =  \mathbf H_u^0 = b_u \mathbf J_u^0 b_u^{-1} = b_u \mathbf H b_u^{-1}$ and so $b_u \in \mathscr N_{\mathbf G}(\mathbf H) = \mathbf L$. Then for $\zeta$-almost every $(u, v) \in W$, we have
$$\varphi(u, v) L = \theta(u, v) L = \theta(u, v) b_u L = \psi(u, v) L.$$
This shows that $p \circ \varphi = p \circ \psi$ $\zeta$-almost everywhere.
\end{proof}

\subsection{Proof of Theorem \ref{main-theorem}}

Let $k$ be a local field of characteristic zero, $\mathbf G$ a $k$-isotropic almost $k$-simple algebraic $k$-group and $\Gamma < \mathbf G(k)$ a Zariski dense discrete subgroup. Let $\Gamma \curvearrowright (X, \nu)$ be a free ergodic pmp action and set $\mathscr R = \mathscr R(\Gamma \curvearrowright X)$. Upon discarding a null Borel subset, we may assume that the Borel action $\Gamma \curvearrowright X$ is free. Consider the orbit Borel $1$-cocycle $\beta : \mathscr R \to \Gamma : (\gamma x, x) \mapsto \gamma$. Let $\mathscr S$ be an ergodic equivalence relation on $(Y, \eta)$ and consider the Borel $1$-cocycle $\alpha = \beta \circ q : \mathscr R \times \mathscr S \to \Gamma$, where $q : \mathscr R \times \mathscr S \to \mathscr R$ is the canonical factor map. Then $\ker(\alpha) = \id_X \times \mathscr S$. 

For every $i \in \{1, 2\}$, let $\mathscr T_i$ be an ergodic equivalence relation on $(Z_i, \zeta_i)$. Set $\mathscr T = \mathscr T_1 \times \mathscr T_2$ and $(Z, \zeta) = (Z_1 \times Z_2, \zeta_1 \otimes \zeta_2)$. Assume that $\mathscr R \times \mathscr S \cong \mathscr T$. Then there exists a conull Borel subset $Z_0 \subset Z$ such that we may identify $(\mathscr R \times \mathscr S)|_{Z_0}$ with $\mathscr T|_{Z_0}$ everywhere and regard $Z_0 \subset X \times Y$ as a conull Borel subset.

We prove Theorem \ref{main-theorem} by contradiction. Assume that for every $i \in \{1, 2\}$, we have $\mathscr T_i \npreceq _{\mathscr T} \mathscr S$ as subequivalences relations of 
$\mathscr T$.

By \cite[Proposition 9.3.2]{Zi84}, for every $i \in \{1, 2\}$, there exists an amenable ergodic subequivalence relation $\mathscr U_i \leq \mathscr T_i$. Set $\mathscr U = \mathscr U_1 \times \mathscr U_2$ and $\mathscr U_0 = \mathscr U|_{Z_0}$ which is an amenable ergodic subequivalence relation of $\mathscr T|_{Z_0}$. 

\begin{claim}\label{claim1}
There exist a proper $k$-subgroup $\mathbf H < \mathbf G$ and a $(\mathscr U, \alpha|_{\mathscr U})$-equivariant measurable map $f : Z \to G/H$, where we regard $G/H \subset (\mathbf G/\mathbf H)(k)$ with $H = \mathbf H(k)$.
\end{claim}

By contradiction, assume that the assertion on Claim \ref{claim1} does not hold. By applying Theorem \ref{thm-non-trivial-gate} to $\alpha|_{\mathscr U_0} : \mathscr U_0 \to G$, there exists a compact subgroup $L < G$ and a $(\mathscr U_0, \alpha|_{\mathscr U_0})$-equivariant measurable map $f : Z_0 \to G/L$. Upon discarding a null Borel subset, we may assume that the map $f : Z_0 \to G/L$ is Borel and strictly $(\mathscr U_0, \alpha|_{\mathscr U_0})$-equivariant. Choose a Borel section $\sigma : G/L \to G$ and define the Borel map $\psi = \sigma \circ f : Z_0 \to G$. Then for every $(z_1, z_2) \in \mathscr U|_{Z_0}$, we have $\psi(z_1) L = \alpha(z_1, z_2) \psi(z_2) L$. Since $G$ is $\sigma$-compact, there exists a compact subset $C \subset G$ such that the Borel subset $V = \psi^{-1}(C) \subset Z_0$ satisfies $\zeta(V) > 0$. Then for every $(z_1, z_2) \in \mathscr U|_V$, we have $\alpha(z_1, z_2) \in \psi(z_1) L \psi(z_2)^{-1} \subset C L C^{-1} \cap \Gamma$. Since $C L C^{-1} \subset G$ is compact and $\Gamma < G$ is discrete, $F = C L C^{-1} \cap \Gamma$ is a finite subset of $\Gamma$. 

Define the subequivalence relation $\mathscr V = \mathscr U|_V \cap \ker(\alpha)|_{V}$. Let $z \in V$. Then we can write $[z]_{\mathscr U|_V} = \bigcup_{\gamma \in F} C_\gamma(z)$, where $C_\gamma(z) = \left \{t \in [z]_{\mathscr U|_V} \mid \alpha(z, t) = \gamma\right \}$ for $\gamma \in F$. Then for every $\gamma \in F$ and all $t_1, t_2 \in C_\gamma(z)$, we have $\alpha(t_1, t_2) = \alpha(t_1, z) \alpha(z, t_2) = \gamma^{-1} \gamma = e$ and so $(t_1, t_2) \in \mathscr V$. Therefore, $\mathscr V \subset \mathscr U|_V$ has bounded index. Since $\mathscr V \leq \ker(\alpha)|_V = (\id_X \times \mathscr S)|_V$, it follows that $\mathscr U \preceq_{\mathscr T} \id_X \times \mathscr S$. Since $\mathscr U_1 \times \id_{Z_2} \leq \mathscr U$, this further implies that $\mathscr U_1 \times \id_{Z_2} \preceq_{\mathscr T} \id_X \times \mathscr S$. Since $\mathscr U_1$ is ergodic, Facts \ref{fact-equivalence}$(\rm v)$ implies that $\mathscr T_1 \times \id_{Z_2} \preceq_{\mathscr T} \id_X \times \mathscr S$. This contradicts our assumption and finishes the proof of Claim \ref{claim1}.

\begin{claim}\label{claim2}
There exist a proper $k$-subgroup $\mathbf L < \mathbf G$ and a $(\mathscr T_1 \times \mathscr T_2, \alpha)$-equivariant measurable map $F : Z \to G/L$, where we regard $G/L \subset (\mathbf G/\mathbf L)(k)$ with $L = \mathbf L(k)$.
\end{claim}

The proof of Claim \ref{claim2} consists of two steps. Firstly, we consider the product equivalence $\mathscr T_1 \times \mathscr U_2$ together with the subequivalence relation $\mathscr U_1 \times \mathscr U_2 \leq \mathscr T_1 \times \mathscr U_2$ and the $(\mathscr U_1 \times \mathscr U_2, \alpha|_{\mathscr U_1 \times \mathscr U_2})$-equivariant measurable map $f : Z \to G/H$. By Theorem \ref{thm-product-algebraic}$(\rm ii)$, we may choose a minimal pair $(\psi, \mathfrak J) \in \mathfrak A(\alpha|_{\mathscr T_1 \times \mathscr U_2})$ for which there exists a Zariski connected $k$-subgroup $\mathbf J < \mathbf G$ such that $\mathbf J = \mathbf J_{z_1}^0$ for $\zeta_1$-almost every $z_1 \in Z_1$. By Claim $\ref{claim1}$, there exist a proper $k$-subgroup $\mathbf H < \mathbf G$ and a $(\mathscr U, \alpha|_{\mathscr U})$-equivariant measurable map $f : Z \to G/H$. Choose a Borel section $\sigma : G/H \to G$. Consider the measurable map $\varphi = \sigma \circ f : Z \to G$ and the constant map $\mathfrak H : Z_1 \to \Sub(G) : z_1 \mapsto H$. Then for $\zeta_1$-almost every $z_1 \in Z_1$, the Borel map $Z_2 \to G/H : z_2 \mapsto \varphi(z_1, z_2)H = f(z_1, z_2)$ is $(\mathscr U_2, \alpha_{z_1})$-equivariant and so we may regard $(\varphi, \mathfrak H) \in  \mathfrak A(\alpha|_{\mathscr T_1 \times \mathscr U_2})$. By minimality of $(\psi, \mathfrak J) \in \mathfrak A(\alpha|_{\mathscr T_1 \times \mathscr U_2})$ and since $\mathbf H < \mathbf G$ is a proper $k$-subgroup, it follows that $\mathbf J < \mathbf G$ is a proper $k$-subgroup. 

Next, we claim that $\mathbf J \neq \{e\}$. Upon discarding a null Borel subset, we may assume that the map $\mathfrak J : Z_1 \to \Sub(G) : z_1 \mapsto J_{z_1}$ is Borel and for every $z_1 \in Z_1$, there exists an algebraic $k$-subgroup $\mathbf J_{z_1} < \mathbf G$ such that $\mathbf J_{z_1}(k) = J_{z_1}$ and $\mathbf J_{z_1}^0 = \mathbf J$. By contradiction, assume that $\mathbf J = \{e\}$. Then we have that $J_{z_1} < G$ is a finite subgroup for every $z_1 \in Z_1$. It is well-known that there are only finitely many conjugacy classes of maximal compact subgroups of $G$ (see e.g.\! \cite[Corollaire 3.3.3]{BT71}). Therefore, there exists a nonnull Borel subset $W \subset Z_1$ and a maximal compact subgroup $K < G$ such that the subset 
$$B = \left\{ (z_1, b) \in W \times G \mid b J_{z_1}b^{-1} < K \right\} \subset W \times G$$
satisfies $\pi_W(B) = W$. Using \eqref{eq-closed-bis}, we have that $B \subset W \times G$ is a Borel subset. By the measurable selection theorem (see Theorem \ref{thm-selection}), there exists a Borel map $W \to G : z_{1} \mapsto b_{z_1}$ such that $b_{z_1} J_{z_1} b_{z_1}^{-1} < K$ for $\zeta_1$-almost every $z_1 \in W$. Upon discarding a null Borel subset, we may assume that the Borel map $W \to G : z_{1} \mapsto b_{z_1}$ satisfies $b_{z_1} J_{z_1} b_{z_1}^{-1} < K$ for every $z_1 \in W$. We may extend the Borel map $Z_1 \to G : z_{1} \mapsto b_{z_1}$ by declaring that $b_{z_1} = e$ for every $z_1 \in Z_1 \setminus W$. Upon replacing the pair $(\psi, \mathfrak J) \in \mathfrak A(\alpha|_{\mathscr T_1 \times \mathscr U_2})$ by the pair $(\rho, \mathfrak K) \in \mathfrak A(\alpha|_{\mathscr T_1 \times \mathscr U_2})$, where $\rho : Z \to G : (z_1, z_2) \mapsto \psi(z_1, z_2)b_{z_1}^{-1}$ and $\mathfrak K : Z_1 \to \Sub(G) : z_1 \mapsto b_{z_1} J_{z_1} b_{z_1}^{-1}$, we may assume that $J_{z_1} < K$ for every $z_1 \in W$. 

Choose a conull Borel subset $V \subset (W \times Z_2) \cap Z_0$ such that the restricted Borel map $\psi|_{V} : V \to G$ satisfies 
$$\forall ((z_1, s), (z_1, t)) \in (\id_{Z_1} \times \mathscr U_2)|_V, \quad \psi(z_1, s)J_{z_1} = \alpha((z_1, s), (z_1, t)) \psi(z_1, t) J_{z_1}.$$ 
Since $G$ is $\sigma$-compact, there exists a compact subset $C \subset G$ such that $V_0 = (\psi|_{V})^{-1}(C) \subset V$ satisfies $\zeta(V_0) > 0$.  Then for every $((z_1, s), (z_1, t)) \in (\id_{Z_1} \times \mathscr U_2)|_{V_0}$, we have 
$$\alpha((z_1, s), (z_1, t)) \in \psi(z_1, s)J_{z_1}\psi(z_1, t)^{-1}  \in C J_{z_1} C^{-1} \cap \Gamma \subset C K C^{-1} \cap \Gamma.$$ Since $C K C^{-1} \subset G$ is compact and $\Gamma < G$ is discrete, $F = C K C^{-1} \cap \Gamma$ is a finite subset of $\Gamma$. Now a reasoning entirely analogous to the one as in the second paragraph of the proof of Claim \ref{claim1} shows that $\id_{Z_1} \times \mathscr U_2 \preceq_{\mathscr T} \id_X \times \mathscr S$. This contradicts our assumption and so $\mathbf J \neq \{e\}$.

Since $\mathbf G$ is almost $k$-simple and since the Zariski connected $k$-subgroup $\mathbf J$ is different from $\{e\}$ and $\mathbf G$, it follows that the $k$-subgroup $\mathbf L = \mathscr N_{\mathbf G}(\mathbf J) < \mathbf G$ is a proper subgroup. Set $L = \mathbf L(k)$. Choose a countable subgroup $\Lambda < [\mathscr T_1]$ such that $\mathscr T_1 = \mathscr R(\Lambda \curvearrowright Z_1)$. For every $\lambda \in \Lambda$, the pair $\lambda_\ast(\psi, \mathfrak J) \in \mathfrak A(\alpha|_{\mathscr T_1 \times \mathscr U_2})$ is minimal and for $\zeta_1$-almost every $z_1 \in Z_1$, we have $(\mathbf J_{\lambda z_1})^0 = \mathbf J$. Then Theorem \ref{thm-product-algebraic}$(\rm iii)$ implies that for $\zeta$-almost every $(z_1, s) \in Z_1 \times Z_2$, we have
$\psi(z_1, s) L = \alpha((z_1, s), (\lambda z_1, s))\psi(\lambda z_1, s) L$. This further implies that for almost every $(z_1, s_1) \in \mathscr T_1$ and almost every $(z_2, s_2) \in \mathscr U_2$, we have
\begin{align*}
\psi(z_1, s_1) L &= \alpha((z_1, s_1), (z_2, s_1))\psi(z_2, s_1) L \\
&=  \alpha((z_1, s_1), (z_2, s_1)) \alpha((z_2, s_1), (z_2, s_2)) \psi(z_2, s_2) L  \\
&=  \alpha((z_1, s_1), (z_2, s_2))\psi(z_2, s_2) L,
\end{align*}
where in the second line we used the fact that $(\psi, \mathfrak J) \in \mathfrak A(\alpha|_{\mathscr T_1 \times \mathscr U_2})$. It follows that the measurable map $Z \to G/L : z \mapsto \psi(z)L$ is $(\mathscr T_1 \times \mathscr U_2, \alpha|_{\mathscr T_1 \times \mathscr U_2})$-equivariant.

Now, we can consider the product equivalence relation $\mathscr T_1 \times \mathscr T_2$ together with the subequivalence relation $\mathscr T_1 \times \mathscr U_2 \leq \mathscr T_1 \times \mathscr T_2$ and the $(\mathscr T_1 \times \mathscr U_2)$-equivariant measurable map $Z \to G/L : z \mapsto \psi(z)L$. Repeating the same argument as in the first four paragraphs of the proof of Claim \ref{claim2} and since $\mathscr T_1 \times \id_{Z_2} \npreceq_{\mathscr T} \id_X \times \mathscr S$, upon modifying the proper $k$-subgroup $\mathbf L < \mathbf G$, we obtain the desired conclusion. This finishes the proof of Claim \ref{claim2}.

Consider now the pmp action $\Gamma \curvearrowright (Z, \zeta)$ defined by $\gamma z = (\gamma x, y)$ for every $z = (x, y) \in Z$. Then for every $\gamma \in \Gamma$ and almost every $z = (x, y) \in Z$, we have
$$F(\gamma z) = F(\gamma x, y) = \alpha((\gamma x, y), (x,y)) F(x, y) = \gamma F(z).$$
Then the measurable map $F : Z \to G/L$ is $\Gamma$-equivariant. Consider the Borel probability measure $\mu = F_\ast\zeta \in \Prob(G/L)^\Gamma$. Since $\Gamma < G$ is Zariski dense in $\mathbf G$, \cite[Proposition 1.9]{BDL14} implies that there exists a normal $k$-subgroup $\mathbf N \lhd \mathbf G$ such that the image of $\Gamma$ in $(\mathbf G/\mathbf N)(k)$ is precompact and $\mu$ is supported on $(\mathbf G/\mathbf L)^{\mathbf N} \cap G/L$. Since $\mathbf L < \mathbf G$ is a proper $k$-subgroup and since $(\mathbf G/\mathbf L)^{\mathbf N} \neq \emptyset$, we have $\mathbf N \neq \mathbf G$. Since $\mathbf G$ is almost $k$-simple, it follows that $\mathbf N \lhd \mathbf G$ is a finite normal $k$-subgroup. Letting $N = \mathbf N(k)$, the image of $\Gamma$ in $(\mathbf G/\mathbf N)(k)$ is contained in $G/N$ which is closed in $(\mathbf G/\mathbf N)(k)$. Then the image of $\Gamma$ in $G/N$ is precompact. Since $N < G$ is finite, this further implies that $\Gamma < G$ is compact, which is a contradiction. Therefore, we have showed that the orbit equivalence relation $\mathscr R(\Gamma \curvearrowright X)$ is strongly prime. This concludes the proof of Theorem \ref{main-theorem}.

The following variation of Theorem \ref{main-theorem} will be useful in the proof of Theorem \ref{main-theorem-decomposition}.

\begin{thm}\label{thm-analogy}
Let $k$ be a local field of characteristic zero, $\mathbf G$ a $k$-isotropic almost $k$-simple algebraic $k$-group and $\Gamma < \mathbf G(k)$ a Zariski dense discrete subgroup. Let $\Lambda$ be a countable discrete group and $\Gamma \times \Lambda \curvearrowright (X, \nu)$ an essentially free ergodic pmp action. Set $\mathscr R = \mathscr R(\Gamma \times \Lambda \curvearrowright X)$.

For every $i \in \{1, 2\}$, let $\mathscr T_i$ be an ergodic equivalence relation on $(Z_i, \zeta_i)$ and assume that $\mathscr R \cong \mathscr T_1 \times \mathscr T_2$. Then there exists $i \in \{1, 2\}$ such that $\mathscr T_i \preceq_{\mathscr R} \mathscr R(\Lambda \curvearrowright X)$ as subequivalence relations.
\end{thm}

\begin{proof}
The proof is identical to the proof of Theorem \ref{main-theorem}, the only difference is that the pmp action $\Gamma \times \Lambda \curvearrowright (X, \nu)$ is no longer assumed to be a product action. To circumvent this issue, we exploit Lemma \ref{lem-key-intertwining}. Upon discarding a null Borel subset, we may assume that the Borel action $\Gamma \times \Lambda \curvearrowright X$ is free. Consider the Borel $1$-cocycle $\alpha : \mathscr R \to \Gamma : ((\gamma, \lambda) x, x) \mapsto \gamma$. Then $\ker(\alpha) = \mathscr R(\Lambda \curvearrowright X)$.

By contradiction, assume that for every $i \in \{1, 2\}$, we have $\mathscr T_i \npreceq_{\mathscr R} \mathscr R(\Lambda \curvearrowright X)$ as subequivalence relations. By Lemma \ref{lem-key-intertwining}, for every $i \in \{1, 2\}$, there exists an amenable ergodic subequivalence relation $\mathscr U_i \leq \mathscr T_i$ such that $\mathscr U_i \npreceq_{\mathscr R} \mathscr R(\Lambda \curvearrowright X)$. Then we can literally repeat the exact same proof as in Claims \ref{claim1} and \ref{claim2} to deduce a contradiction.
\end{proof}

We use Theorem \ref{thm-analogy} in combination with Connes--Jones'\! construction \cite{CJ81} to provide examples of type ${\rm II_1}$ ergodic equivalence relations $\mathscr R$ that are prime but for which the associated type ${\rm II_1}$ factor $\rL(\mathscr R)$ is not prime.

\begin{prop}\label{prop-counterexample}
Let $\Gamma = \SL_3(\Z)$ and $H$ be a nonabelian finite discrete group. Set $\Lambda = \bigoplus_\N H$, $(Y, \eta) = (H^\N, \Haar)$ and $(X, \nu) = (Y^\Gamma, \eta^{\otimes \Gamma})$. Consider the free ergodic pmp action $\Gamma \times \Lambda \curvearrowright (X, \nu)$, where $\Lambda \curvearrowright (X, \nu)$ acts diagonally and $\Gamma \curvearrowright (X, \nu)$ acts by Bernoulli shifts. Set $\mathscr R = \mathscr R(\Gamma \times \Lambda \curvearrowright X)$.

Then $\mathscr R$ is prime but $\rL(\mathscr R)$ is not prime. In fact, $\rL(\mathscr R)$ is a McDuff factor meaning that $\rL(\mathscr R) \cong \rL(\mathscr R) \ovt R$, where $R$ is the unique hyperfinite type ${\rm II_1}$ factor.
\end{prop}

\begin{proof}
By Connes--Jones'\! result \cite{CJ81}, $\rL(\mathscr R)$ is a McDuff factor. Since $\Gamma$ has Kazhdan's property (T) and since $\Gamma \curvearrowright (X, \nu)$ is ergodic, it follows that $\mathscr R$ is strongly ergodic. Assume that $\mathscr R \cong \mathscr T_1 \times \mathscr T_2$, where $\mathscr T_1$, $\mathscr T_2$ are ergodic equivalence relations. By Theorem \ref{thm-analogy}, there exists $i \in \{1, 2\}$ such that $\mathscr T_i \preceq_{\mathscr R} \mathscr R(\Lambda \curvearrowright X)$ as subequivalence relations. Since $\mathscr R(\Lambda \curvearrowright X)$ is amenable, it follows that $\mathscr T_i$ is amenable and so $\mathscr T_i$ is hyperfinite by \cite{CFW81}. Since $\mathscr R$ is strongly ergodic, it follows that $\mathscr T_i$ is finite. Therefore, $\mathscr R$ is prime.
\end{proof}

\subsection{Proof of Corollary \ref{main-cor}}

Let $n \geq 1$ and $\mathscr R_1, \dots, \mathscr R_n$ be ergodic equivalence relations that belong to $\mathfrak Z$. For every $i \in \{1, \dots, n\}$, write $\mathscr R_i = \mathscr R(\Gamma_i \curvearrowright X_i)$.

$(\rm i)$ Assume that $\mathscr R = \mathscr R_1 \times \cdots \times \mathscr R_n \cong \mathscr T_1 \times  \mathscr T_2$, where $ \mathscr T_1$ and $\mathscr T_2$ are type ${\rm II_1}$ ergodic equivalence relations acting on $(Y_1, \eta_1)$ and $(Y_2, \eta_2)$ respectively. For every $j \in \{1, 2\}$, denote by $T_j \subset \{1, \dots, n\}$ a minimal subset for which $\mathscr T_j \preceq_{\mathscr R} \mathscr R_{T_j}$ as subequivalence relations, where $\mathscr R_{T_j} = \prod_{i \in T_j} \mathscr R_i$. Since $\mathscr T_j$ is a type ${\rm II_1}$ ergodic equivalence relation, we have $T_j \neq \emptyset$. Set $T = T_1 \cup T_2$ and $\mathscr R_T = \prod_{i \in T} \mathscr R_i$. 

Firstly, we claim that $T = \{1, \dots, n\}$. Indeed, we have $\mathscr T_1 \preceq_{\mathscr R} \mathscr R_T$ and $\mathscr T_2 \preceq_{\mathscr R} \mathscr R_T$ as subequivalence relations. By Facts \ref{fact-equivalence}$(\rm i)$, we obtain that $\rL(\mathscr T_1) \ovt \rL^\infty(Y_2) \preceq_{\rL(\mathscr R)} \rL(\mathscr R_T) \ovt \rL^\infty(X_{T^c})$ and $\rL^\infty(Y_1) \ovt \rL(\mathscr T_2) \preceq_{\rL(\mathscr R)} \rL(\mathscr R_T) \ovt \rL^\infty(X_{T^c})$. By \cite[Lemma 2.4(3)]{DHI16}, we moreover have $\rL(\mathscr T_1) \ovt \rL^\infty(Y_2) \preceq^s_{\rL(\mathscr R)} \rL(\mathscr R_T) \ovt \rL^\infty(X_{T^c})$ and $\rL^\infty(Y_1) \ovt \rL(\mathscr T_2) \preceq^s_{\rL(\mathscr R)} \rL(\mathscr R_T) \ovt \rL^\infty(X_{T^c})$. Then \cite[Lemma 2.3]{BV12} implies that $\rL(\mathscr R) \preceq_{\rL(\mathscr R)} \rL(\mathscr R_T) \ovt \rL^\infty(X_{T^c})$. Since the group $\Gamma_i$ is infinite and the pmp action $\Gamma_i \curvearrowright (X_i, \nu_i)$ is essentially free ergodic for every $i \in \{1, \dots, n\}$, \cite[Theorem 4.4]{HI17} implies that we necessarily have $T = \{1, \dots, n\}$.

Secondly, we claim that $T_1 \cap T_2 = \emptyset$. By contradiction, assume that $T_1 \cap T_2 \neq \emptyset$ and let $i \in T_1 \cap T_2$. Write $\mathscr R = \mathscr R_i \times \mathscr S$, where $\mathscr S = \mathscr R_{\{1, \dots, n\} \setminus \{i\}}$. By Theorem \ref{main-theorem}, we have $\mathscr T_1 \preceq_{\mathscr R} \mathscr S$ or $\mathscr T_2 \preceq_{\mathscr R} \mathscr S$ as subequivalence relations. Assume that $\mathscr T_1 \preceq_{\mathscr R} \mathscr S$ as subequivalence relations. The reasoning in the previous paragraph shows that $\rL(\mathscr T_1) \ovt \rL^\infty(Y_2) \preceq^s_{\rL(\mathscr R)} \rL(\mathscr S) \ovt \rL^\infty(X_i)$ and $\rL(\mathscr T_1) \ovt \rL^\infty(Y_2) \preceq^s_{\rL(\mathscr R)} \rL(\mathscr R_{T_1}) \ovt \rL^\infty(X_{T_1^c})$. Then \cite[Lemma 2.8(2)]{DHI16} implies that $\rL(\mathscr T_1) \ovt \rL^\infty(Y_2) \preceq^s_{\rL(\mathscr R)} \rL(\mathscr R_{T_1 \setminus \{i\}}) \ovt \rL^\infty(X_{(T_1 \setminus \{i\})^c })$. Then Fact \ref{fact-equivalence}$(\rm i)$ implies that $\mathscr T_1 \preceq_{\mathscr R} \mathscr R_{T_1 \setminus \{i\}}$ as subequivalence relations. This however contradicts the minimality of the set $T_1$. Likewise, $\mathscr T_2 \preceq_{\mathscr R} \mathscr S$ contradicts the minimality of the set $T_2$. Therefore, we have $T_1 \cap T_2 = \emptyset$.

Therefore, we may write $\{1, \dots, n\} = T_1 \sqcup T_2$. We have $\mathscr R = \mathscr R_{T_1} \times \mathscr R_{T_2} \cong \mathscr T_1 \times \mathscr T_2$. Moreover, we have $\mathscr T_1 \preceq_{\mathscr R} \mathscr R_{T_1}$ and $\mathscr T_2 \preceq_{\mathscr R} \mathscr R_{T_2}$ as subequivalence relations.

$(\rm ii)$ We prove the assertion by complete induction over $n \geq 1$. Theorem \ref{main-theorem} implies that the assertion holds for $n = 1$. Assume that the assertion holds for every $1 \leq j \leq n - 1$ and let us prove that it holds for $n$. 

Firstly, assume that $\mathscr R = \mathscr R_1 \times \cdots \times \mathscr R_{n} \cong \mathscr T_1 \times \cdots \times \mathscr T_p$, where $\mathscr T_1, \dots, \mathscr T_p$ are type ${\rm II_1}$ ergodic equivalence relations. We may assume that $p \geq n \geq 2$. Then item $(\rm i)$ implies that there exists a partition $\{1, \dots, n\} = T_1 \sqcup T_2$ into nonempty subsets such that $\mathscr T_1 \preceq_{\mathscr R} \mathscr R_{T_1}$ and $\mathscr T_2 \times \cdots \times \mathscr T_p \preceq_{\mathscr R} \mathscr R_{T_2}$ as subequivalence relations. By Facts \ref{fact-equivalence}$(\rm vi)$, there exists an ergodic equivalence relation $\mathscr V$ such that $\mathscr T_2 \times \cdots \times \mathscr T_p \times \mathscr V$ and $\mathscr R_{T_2}$ are stably isomorphic. Upon replacing $\mathscr T_2$ by $\mathscr T_2^t$ for $t > 0$, we may assume that $\mathscr T_2 \times \cdots \times \mathscr T_{p} \times \mathscr V \cong \mathscr R_{T_2}$. By induction hypothesis and since $p - 1 \geq |T_2|$ and $|T_2| \leq n- 1$, we have that $\mathscr V$ is finite, $p - 1 = |T_2| \leq n - 1$ and so $p = n$.

Secondly, assume that $\mathscr R = \mathscr R_1 \times \cdots \times \mathscr R_{n} \cong \mathscr T_1 \times \cdots \times \mathscr T_{n}$. Then item $(\rm i)$ implies that there exists a partition $\{1, \dots, n\} = T_1 \sqcup T_2$ into nonempty subsets such that $\mathscr T_1 \preceq_{\mathscr R} \mathscr R_{T_1}$ and $\mathscr T_2 \times \cdots \times \mathscr T_n \preceq_{\mathscr R} \mathscr R_{T_2}$ as subequivalence relations. By Facts \ref{fact-equivalence}$(\rm vi)$, there exists an ergodic equivalence relation $\mathscr U$ such that $\mathscr T_1 \times \mathscr U$ and $\mathscr R_{T_1}$ are stably isomorphic and there exists an ergodic equivalence relation $\mathscr V$ such that $\mathscr T_2 \times \cdots \times \mathscr T_n \times \mathscr V$ and $\mathscr R_{T_2}$ are stably isomorphic. Upon replacing $\mathscr T_2$ by $\mathscr T_2^t$ for $t > 0$, we may assume that $\mathscr T_2 \times \cdots \times \mathscr T_{n} \times \mathscr V \cong \mathscr R_{T_2}$. Since $|T_2| \leq n - 1$, the previous paragraph implies that $\mathscr V$ is finite, $|T_2| = n - 1$ and so $|T_1| = 1$. Upon replacing $\mathscr T_n$ by $\mathscr T_n \times \mathscr V$, we may assume that $\mathscr V$ is trivial. Upon permuting the indices, we may assume that $T_1 = \{1\}$ and so $\mathscr T_1 \times \mathscr U$ and $\mathscr R_1$ are stably isomorphic. Since $\mathscr R_1$ is strongly prime, we have that $\mathscr U$ is finite and so $\mathscr T_1$ and $\mathscr R_1$ are stably isomorphic. Moreover, we have $\mathscr R_2 \times \cdots \times \mathscr R_{n} \cong \mathscr T_2 \times \cdots \times \mathscr T_{n}$. By induction hypothesis and upon permuting the indices, we have that $\mathscr T_i$ and $\mathscr R_i$ are stably isomorphic for every $i \in \{2, \dots, n\}$. This concludes the proof of Corollary \ref{main-cor}.

\subsection{Proof of Corollary \ref{main-cor-sharp}}

Let $n \geq 1$ and $\mathscr R_1, \dots, \mathscr R_n$ be strongly ergodic equivalence relations that belong to $\mathfrak Z$. For every $i \in \{1, \dots, n\}$, write $\mathscr R_i = \mathscr R(\Gamma_i \curvearrowright X_i)$. For every $i \in \{1, \dots, n\}$, since $\Gamma_i$ is not inner amenable by Proposition \ref{prop-not-inner-amenable}, the type ${\rm II_1}$ factor $\rL(\mathscr R_i)$ is necessarily full by \cite{Ch81}.

$(\rm i)$ Assume that $\mathscr R = \mathscr R_1 \times \cdots \times \mathscr R_n \cong \mathscr T_1 \times \mathscr T_2$, where $\mathscr T_1$ and $\mathscr T_2$ are type ${\rm II_1}$ ergodic equivalence relations acting on $(Y_1, \eta_1)$ and $(Y_2, \eta_2)$ respectively. We may literally repeat the proof of Corollary \ref{main-cor}$(\rm i)$. There exists a partition $\{1, \dots, n\} = T_1 \sqcup T_2$ into nonempty subsets such that $\rL(\mathscr T_1) \ovt \rL^\infty(Y_2) \preceq^s_{\rL(\mathscr R)} \rL(\mathscr R_{T_1}) \ovt \rL^\infty(X_{T_2})$ and $\rL^\infty(Y_1) \ovt \rL(\mathscr T_2) \preceq^s_{\rL(\mathscr R)} \rL^\infty(X_{T_1}) \ovt \rL(\mathscr R_{T_2})$. A combination of \cite[Lemma 2.6(3)]{DHI16} and \cite[Proposition 2.4(3)]{OP07} implies that $\rL(\mathscr T_j)$ is amenable relative to $\rL(\mathscr R_{T_j})$ inside $\rL(\mathscr R)$ for every $j \in \{1, 2\}$. Note that $\rL(\mathscr R)$ is full by \cite[Corollary 2.3]{Co75} and so are $\rL(\mathscr T_1)$, $\rL(\mathscr T_2)$, $\rL(\mathscr R_{T_1})$, $\rL(\mathscr R_{T_2})$. Then \cite[Lemma 5.2]{IM19} implies that $\rL(\mathscr T_j) \preceq_{\rL(\mathscr R)} \rL(\mathscr R_{T_j})$ for every $j \in \{1, 2\}$. Reasoning as in the proof of \cite[Corollary E(1)]{Ho15}, we infer that there exists $t > 0$ such that $\mathscr T_1^t \cong \mathscr R_{T_1}$ and $\mathscr T_2^{1/t} \cong \mathscr R_{T_2}$.

$(\rm ii)$ The proof follows by combining the ones of item $(\rm i)$, Corollary \ref{main-cor}$(\rm ii)$ and \cite[Corollary E]{Ho15}. This concludes the proof of Corollary \ref{main-cor-sharp}.

\subsection{Proof of Theorem \ref{main-theorem-decomposition}}

We prove the following key intermediate result which is analogous to \cite[Proposition 4.1]{Dr19}.

\begin{prop}\label{prop-key}
Let $n \geq 1$. For every $1 \leq i \leq n$, let $k_i$ be a local field of characteristic zero, $\mathbf G_i$ a $k_i$-isotropic $k_i$-simple algebraic $k_i$-group and $\Gamma_i < \mathbf G(k)$ a Zariski dense discrete subgroup with Kazhdan's property {\em (T)}. Set $\Gamma = \Gamma_1 \times \cdots \times \Gamma_n$. Let $\Gamma \curvearrowright (X, \nu)$ be an essentially free ergodic pmp action and set $\mathscr R = \mathscr R(\Gamma \curvearrowright X)$. Assume that $\mathscr R \cong \mathscr T_1 \times \mathscr T_2$, where $\mathscr T_1, \mathscr T_2$ are type ${\rm II_1}$ ergodic equivalence relations. Set $M = \rL(\mathscr R)$ and for every $j \in \{1, 2\}$, set $P_j = \rL(\mathscr T_j)$.

Then there exist a partition $\{1, \dots, n\} = T_1 \sqcup T_2$ into nonempty subsets such that $P_j \preceq^s_M \rL(\Gamma_{T_j} \curvearrowright X)$ for every $j \in \{1, 2\}$.
\end{prop}

\begin{proof}
For every $j \in \{1, 2\}$, denote by $T_j \subset \{1, \dots, n\}$ a minimal subset for which $P_j \preceq^s_{M} M_{T_j}$, where $M_{T_j} = \rL(\Gamma_{T_j} \curvearrowright X)$ and $\Gamma_{T_j} = \prod_{i \in T_j} \Gamma_i$. Since $P_j$ is a type ${\rm II_1}$ factor, we have $T_j \neq \emptyset$. Set $T = T_1 \cup T_2$, $\Gamma_T = \prod_{i \in T} \Gamma_i$ and $M_T =\rL(\Gamma_T \curvearrowright X)$. Arguing as in the proof of Corollary \ref{main-cor}, we deduce that $T = \{1, \dots, n\}$.

It remains to prove that $T_1 \cap T_2 = \emptyset$. By contradiction, assume that $T_1 \cap T_2 \neq \emptyset$ and let $i \in T_1 \cap T_2$. Set $\Lambda = \prod_{j \neq i} \Gamma_j$ and $\mathscr S = \mathscr R(\Lambda \curvearrowright X)$. Arguing as in the proof of Corollary \ref{main-cor}, by minimality of $T_1$ and $T_2$ and \cite[Lemma 2.8(2)]{DHI16}, we have that $P_j \npreceq_M^s \rL(\Lambda \curvearrowright X)$ for every $j \in \{1, 2\}$. By Fact \ref{fact-equivalence}$(\rm i)$, this further implies that $\mathscr T_j \npreceq_{\mathscr R} \mathscr S$ as subequivalence relations. This however contradicts Theorem \ref{thm-analogy}. Therefore, we have $\{1, \dots, n\} = T_1 \sqcup T_2$.
 \end{proof}

\begin{proof}[Proof of Theorem \ref{main-theorem-decomposition}]
By combining Proposition \ref{prop-key} and \cite[Theorem 3.2]{Dr19}, we infer that there exist a partition $\{1, \dots, n\} = T_1 \sqcup T_2$ into nonempty subsets, $t > 0$, an essentially free ergodic pmp action $\Gamma_{T_j} \curvearrowright (X_j, \nu_j)$, where $\Gamma_{T_j} = \prod_{i \in T_j} \Gamma_i$ and $\rL^\infty(X_j, \nu_j) = \rL^\infty(X, \nu)^{\Gamma_{T_{j + 1}}}$ for every $j \in \Z/2\Z$, a decomposition $\rL(\mathscr R) = \rL(\mathscr T_1)^t  \ovt \rL(\mathscr T_2)^{1/t}$ and $u \in \mathscr U(\rL(\mathscr R))$, such that $\Gamma \curvearrowright (X, \nu)$ is isomorphic to $\Gamma_{T_1} \times \Gamma_{T_2} \curvearrowright (X_1 \times X_2, \nu_1 \otimes \nu_2)$ and
$$\rL(\mathscr T_1)^t = u\rL(\Gamma_{T_1} \curvearrowright X_1)u^* \quad \text{and} \quad \rL(\mathscr T_2)^{1/t} = u \rL(\Gamma_{T_2} \curvearrowright X_2)u^*.$$
Arguing as in the proof of \cite[Corollary E]{Ho15}, we infer that 
$$\mathscr T_1^t \cong \mathscr R(\Gamma_{T_1} \curvearrowright X_1) \quad \text{and} \quad \mathscr T_2^{1/t} \cong \mathscr R(\Gamma_{T_2} \curvearrowright X_2).$$
This finishes the proof of Theorem \ref{main-theorem-decomposition}.
\end{proof}

\subsection{Proof of Corollary \ref{main-cor-decomposition}}
By applying Theorem \ref{main-theorem-decomposition} finitely many times, we can find a partition $\{1, \dots, n\} = T_1 \sqcup \cdots \sqcup T_r$ into nonempty subsets, an essentially free ergodic pmp action $\Gamma_{T_j} \curvearrowright (X_j, \nu_j)$, where $\Gamma_{T_j} = \prod_{i \in T_j} \Gamma_i$ for every $j \in \{1, \dots, r\}$, such that $\Gamma \curvearrowright (X, \nu)$ is isomorphic to $\Gamma_{T_1} \times \cdots \times \Gamma_{T_r} \curvearrowright (X_1 \times \cdots \times X_r, \nu_1 \otimes \cdots \otimes \nu_r)$ and $\mathscr R(\Gamma_{T_j} \curvearrowright X_j)$ is prime for every $j \in \{1, \dots, r\}$.

To prove the uniqueness part, we consider another partition $\{1, \dots, n\} = S_1 \sqcup \cdots \sqcup S_p$ into nonempty subsets, an essentially free ergodic pmp action $\Gamma_{S_j} \curvearrowright (Y_j, \nu_j)$, where $\Gamma_{S_j} = \prod_{i \in S_j} \Gamma_i$ for every $j \in \{1, \dots, p\}$ such that $\Gamma \curvearrowright (X, \nu)$ is isomorphic to $\Gamma_{S_1} \times \cdots \times \Gamma_{S_p} \curvearrowright (Y_1 \times \cdots \times Y_p, \eta_1 \otimes \cdots \otimes \eta_p)$ and $\mathscr R(\Gamma_{S_j} \curvearrowright Y_j)$ is prime for every $j \in \{1, \dots, p\}$. Without loss of generality, we may assume that $r\ge p$. We claim that for any $1\leq j\leq r$, there is $1\leq \ell \leq p$ such that $T_j\subset S_\ell$. To see this, note that for any $1\leq j\leq r$, there is $1\leq \ell \leq p$ such that $S_\ell\cap T_j\neq\emptyset$. Note that the action $\Gamma_{T_j}\curvearrowright (X,\nu)$ is isomorphic to $\Gamma_{S_1\cap T_j} \times \cdots \times \Gamma_{S_p\cap T_j} \curvearrowright (Y_1 \times \cdots \times Y_p, \eta_1 \otimes \cdots \otimes \eta_p)$. If $T_j\not\subset S_\ell$, it is easy to see that $\mathscr R (\Gamma_{T_j}\curvearrowright X_j)$ is not prime, contradiction.

The above claim implies that $r=p$, and hence, $T_j=S_j$, for any $1\leq j\leq r$. Then the conclusion follows.

\bibliographystyle{plain}

\end{document}